\documentclass[10pt]{amsart}
\usepackage{amsxtra, amsfonts, amsmath, amsthm, amstext, amssymb, amscd, mathrsfs, verbatim, color}
\usepackage{threeparttable}
\usepackage{mathtools}
\usepackage[ansinew]{inputenc}\usepackage[T1]{fontenc}
\usepackage[all,cmtip]{xy}
\usepackage[ngerman,french,english]{babel}

\theoremstyle{plain}

\newtheorem{theorem}{Theorem}[section]
\newtheorem{lemma}[theorem]{Lemma}
\newtheorem{definition-theorem}[theorem]{Definition-Theorem}
\newtheorem{proposition}[theorem]{Proposition}
\newtheorem{corollary}[theorem]{Corollary}

\newtheorem{conjecture}[theorem]{Conjecture}

\theoremstyle{definition}
\newtheorem{definition}[theorem]{Definition}
\newtheorem{example}[theorem]{Example}
\newtheorem{remark}[theorem]{Remark}
\newtheorem{notation}[theorem]{Notation}
\newcommand \bth[1] { \begin{theorem}\label{t#1} }
\newcommand \ble[1] { \begin{lemma}\label{l#1} }

\newcommand \bpr[1] { \begin{proposition}\label{p#1} }
\newcommand \bco[1] { \begin{corollary}\label{c#1} }
\newcommand \bde[1] { \begin{definition}\label{d#1}\rm }
\newcommand \bex[1] { \begin{example}\label{e#1}\rm }
\newcommand \bre[1] { \begin{remark}\label{r#1}\rm }

\newcommand \bnota[1] {\begin{notation}\label{n#1}\rm }
\newcommand {\ele} { \end{lemma} }

\newcommand {\epr} { \end{proposition} }
\newcommand {\eco} { \end{corollary} }
\newcommand {\ede} { \end{definition} }
\newcommand {\eex} { \end{example} }
\newcommand {\ere} { \end{remark} }
\newcommand {\enota} { \end{notation} }

\begin{document}
\setlength{\baselineskip}{1.2\baselineskip}
\title[Non-tempered Gan-Gross-Prasad conjecture]
{Restriction for General Linear Groups: the local non-tempered Gan-Gross-Prasad conjecture (non-Archimedean case)  }
\author[Kei Yuen Chan]{Kei Yuen Chan}


\address{
Shanghai Center for Mathematical Sciences \\
Fudan University}
\email{kychan@fudan.edu.cn}

\maketitle

\selectlanguage{english} 

\begin{abstract}
We prove a local Gan-Gross-Prasad conjecture on predicting the branching law for the non-tempered representations of general linear groups in the case of non-Archimedean fields. We also generalize to Bessel and Fourier-Jacobi models and study a possible generalization to Ext-branching laws.\end{abstract}

\section{Introduction}

In 1990s', Gross-Prasad \cite{GP93} formulated conjectures which determine when an irreducible generic representation of $\mathrm{SO}_{n-1}(F)$ appears in a quotient of an irreducible generic representation of $\mathrm{SO}_n(F)$, where $F$ is a local field. The conjectural answer is in terms of symplectic root numbers, providing deep connections with number theory. About ten years ago, Gan-Gross-Prasad \cite{GGP12} generalized the conjectures to other classical groups. For $p$-adic groups, the local generic conjectures in orthogonal, unitary and symplectic-metaplectic cases have been respectively settled by Waldspurger \cite{Wa12}, M{\oe}glin-Waldspurger \cite{MW12}, and by Beuzart-Plessis \cite{BP14}, Gan-Ichino \cite{GI16}, and by Atobe \cite{At18}; and for real groups, the unitary cases for tempered representations and independently for discrete series are settled by Beuzart-Plessis \cite{BP15} and H. He \cite{He17} respectively. We remark that the generic case for general linear groups has been known long from the work of Jacquet--Piateski-Shapiro--Shalika \cite{JPSS83}.

Recently, Gan-Gross-Prasad \cite{GGP19} formulated new conjectures for certain nontempered representations arising from a local component of an automrophic representation. The main goal of this paper is to prove one of those conjectures for general linear groups over a non-Archimedean local field and study related generalizations.






\subsection{Local non-tempered Gan-Gross-Prasad conjecture}

We begin with a precise formulation of the non-tempered conjecture. Let $G_n=\mathrm{GL}_n(F)$, the general linear group over a local field $F$. 

Let $W_F$ be the Weil group of $F$. The Weil-Deligne group $WD_F$ of $F$ is defined as:
\[ WD_F =\left\{ \begin{array}{l l} W_F \times \mathrm{SL}_2(\mathbb C) & \mbox{ if $F$ is non-Archimedean } \\
      W_F & \mbox{ if $F$ is Archimedean } \end{array} \right.
\]
The set of Langlands parameters of $G_n$ is the set of equivalence classes of homomorphisms
\[  \phi: WD_F \rightarrow {}^LG= \mathrm{GL}_n(\mathbb C) ,
\]
under conjugation by elements in $\mathrm{GL}_n(\mathbb C)$, and the restriction to the factor of $\mathrm{SL}_2(\mathbb{C})$ in $W_F$ is algebraic. The local Langlands correspondence for $\mathrm{GL}_n(F)$ is now known by \cite{Ze80,LRS93, HT01, He00, Sc13}.

Define the Arthur parameters \cite{Ar89} as the set of ${}^LG$-orbits of maps 
\[  \psi: WD_F\times \mathrm{SL}_2(\mathbb C) \rightarrow {}^LG
\]
such that $\psi|_{WD_F}$ has bounded image i.e. has tempered Langlands parameter, and the restriction to the $\mathrm{SL}_2(\mathbb C)$ factor is algebraic. For each Arthur parameter $\psi$, one assigns a $L$-parameter given by
\[  \phi_{\psi}(w) =\psi(w, \begin{pmatrix} |w|^{1/2} & 0 \\ 0 & |w|^{-1/2} \end{pmatrix}) .
\]

Let $\mathrm{Sym}^k(\mathbb C^2)$ be the unique $(k+1)$-dimensional irreducible representation of $\mathrm{SL}_2(\mathbb C)$. The Arthur parameter, as a finite $WD_F  \times \mathrm{SL}_2(\mathbb{C})$-representation $\psi$, takes the form
\begin{align} \label{eqn arthur para form}
 M_A= \sum_d M_d \otimes \mathrm{Sym}^{d}(\mathbb{C}^2) ,
\end{align}
where each $M_d$ is a representation of $WD_F$ such that $\psi|_{WD_F}$ has bounded image i.e. each $M_i$ corresponds to a tempered representation. It gives rise to a Langlands parameter $M$ as described above, and gives a $G_n$-representation denoted by $\pi_M$. Any irreducible smooth representation of $G_n$ associated to the Langlands parameter $\phi_{\psi}$ coming from an Arthur parameter is called a representation of Arhtur type.

A key notion in \cite{GGP19} is the relevant pair which governs the branching law of representations of Arthur type: 

\begin{definition} \label{def relevant pair}  \cite{GGP19}
For any $n,m \in \mathbb Z_{\geq 0}$, two Arthur parameters $M_A$ and $N_A$ for respective $G_n$ and $G_m$ are said to form a {\it relevant} pair if there exists $WD_F$-representations $M_0^+, \ldots, M_r^+, M_0^-, \ldots, M_s^-$ (possibly zero) corresponding to tempered representations such that 
\begin{align} \label{eqn relevant MA}
 M_A = \sum_{d=0}^r M_{d}^+ \otimes \mathrm{Sym}^d(\mathbb{C}^2) \oplus \sum_{d=1}^s M_{d}^- \otimes \mathrm{Sym}^{d-1}(\mathbb{C}^2) ,
\end{align}
and
\begin{align} \label{eqn relevant NA}
 N_A=\sum_{d=1}^r M_d^+ \otimes \mathrm{Sym}^{d-1}(\mathbb{C}^2) \oplus \sum_{d=0}^s M_d^-\otimes \mathrm{Sym}^{d}(\mathbb{C}^2)  .
\end{align}

\end{definition}

We remark that in the above definition, the dimensions of the Arthur parameters $M_A$ and $N_A$ are not required to be of corank $1$.



We regard $G_n$ as a subgroup of $G_{n+1}$ via the embedding $g \mapsto \mathrm{diag}(g ,1 )$. A non-tempered Gan-Gross-Prasad conjecture predicts which Arthur type representations of $G_n$ appear in the quotient of an Arthur type representation of $G_{n+1}$, in terms of relevant pairs.

\begin{conjecture} \label{conj ggp orig} \cite[Conjecture 5.1]{GGP19}
Let $F$ be a local field. Let $\pi_M$ and $\pi_N$ be Arthur type representations of $\mathrm{GL}_{n+1}(F)$ and $\mathrm{GL}_n(F)$ respectively. Then $\mathrm{Hom}_{G_n}(\pi_M, \pi_N) \neq 0$ if and only if  their respective associated Arthur parameters $M_A$ and $N_A$ are relevant.
\end{conjecture}

The main result of the paper is to prove the conjecture for non-Archimedean field $F$. Previously, for non-Archimedean $F$, certain cases including when the Deligne $\mathrm{SL}_2(\mathbb{C})$ in $WD_F$ acts trivially are proved in \cite{GGP19}, and M. Gurevich \cite{Gu18} proves the only if direction. We shall give another proof for the only if direction in this paper. Recently, Gourevitch-Sayag \cite{GS20} have results towards the Archimedean case. The unitary restriction problem is studied in \cite{Ve05} by Venkatesh.

\begin{theorem}
If $F$ is non-Archimedean, Conjecture \ref{conj ggp orig} holds.
\end{theorem}

\subsection{Representation-theoretic reformulation} \label{ss rep the reform}

From now on, we assume $F$ is non-Archimedean. Let $\mathrm{Alg}(G_n)$ be the category of smooth $G_n$-representations. We first reformulate the problem into a representation theory setup.

For representations $\pi_i$ in $\mathrm{Alg}(G_{n_i})$ $(i=1,\ldots, k)$ and $n=n_1+\ldots +n_k$, we define the product
\[  \pi_1 \times \ldots \times \pi_k \in \mathrm{Alg}(G_n)
\]
to be the normalized parabolic parabolically induced module from $\pi_1 \boxtimes \ldots \boxtimes \pi_k$.  For more detailed notions of Zelevinsky segments and product, see Section \ref{sec zel segments}. For an irreducible unitarizable cuspidal representation $\rho$ of $G_l$, let 
\[ \Delta_{\rho}(m)=[ \nu^{-(m-1)/2}\rho, \nu^{(m-1)/2}\rho] \]
 be a Zelevinsky segment. Any square integrable representation is known to be isomorphic to $\mathrm{St}(\Delta_{\rho}(m))$ for some such Zelevinsky segment $\Delta_{\rho}(m)$ \cite{Ze80}.  Any tempered representation is isomorphic to a product of some square-integrable representations, and corresponds to a $WD_F$-representation $\psi$ with bounded image $\psi(WD_F)$.

Let $v_{\rho}(m,d)$ be the unique irreducible quotient of the product 
\[  \mathrm{St}(\nu^{(d-1)/2}\Delta_{\rho}(m)) \times \ldots \times \mathrm{St}(\nu^{-(d-1)/2}\Delta_{\rho}(m)) ,
\]
which is so-called a Speh representation and is unitarizable. Each factor $M_d \otimes \mathrm{Sym}^d(\mathbb{C})$ in (\ref{eqn arthur para form}) corresponds to a product of Speh representations of the form 
\begin{align} \label{eqn speh product} v_{\rho_1}(m_1, d+1) \times \ldots \times v_{\rho_r}(m_r, d+1) .
\end{align} 
Any Arthur type representation is a product of some Speh representations. It follows from \cite{Be84, Ta86} that such product is irreducible, and is independent of the ordering of Speh representations.

The notion of a derivative is defined in \cite{Ze80} (see Section \ref{ss bz functor} for the detail). For an irreducible $\pi \in \mathrm{Alg}(G_r)$, let $\widetilde{\pi}$ be the highest derivative of $\pi$ and let $\pi^-=\nu^{1/2}\widetilde{\pi}$, where $\nu(g)=|\mathrm{det}g|_F$. A key observation in \cite{GGP19} is that
\begin{align} \label{eqn derivative unitary}
  v_{\rho}(m,d+1)^- \cong v_{\rho}(m,d) ,
\end{align}
(when $d=0$, we regard $v_{\rho}(m,0)=1$) and so
\begin{align} \label{eqn derivative product}
 (v_{\rho_1}(m_1, d+1)\times \ldots \times v_{\rho_r}(m_r, d+1) )^- \cong v_{\rho_1}(m_1,d)\times \ldots \times v_{\rho_r}(m_r, d) ,
\end{align}
which is also a motivation for the notion of relevant pairs in \cite{GGP19}. The isomorphism (\ref{eqn derivative unitary}) follows from the well-known highest derivative of Zelevinsky \cite{Ze80} (and its translation to the Zelevinsky classification via \cite{Ta86}).

Thus combining Definition \ref{def relevant pair}, (\ref{eqn speh product}) and (\ref{eqn derivative product}), we have the following reformulation: \\

\noindent
{\bf Reformulation of Conjecture \ref{conj ggp orig} for non-Archimedean.} Let $F$ be a non-Archimedean local field. Let $\pi_M$ and $\pi_N$ be Arthur type representations of $\mathrm{GL}_{n+1}(F)$ and $\mathrm{GL}_n(F)$ respectively. Then $\mathrm{Hom}_{G_n}(\pi_M, \pi_N) \neq 0$ if and only if there exist Speh representations $\pi_{p,1}, \ldots, \pi_{p,r}$ and $\pi_{q,1}, \ldots, \pi_{q,s}$ such that
\[ \pi_M \cong \pi_{p,1}\times \ldots \times \pi_{p,r}\times \pi_{q,1}^-\times \ldots \times \pi_{q,s}^- 
\]
and
\[ \pi_N \cong \pi_{p,1}^-\times \ldots \times \pi_{p,r}^- \times \pi_{q,1} \times \ldots \times \pi_{q,s} .
\]

\subsection{Generalizations} \label{ss intro generalize}

The first generalization is on Bessel and Fourier-Jacobi models (Theorem \ref{thm general cases}). Such generalization is also expected in \cite{GGP19}. The strategy for proving general cases is connecting those models functorially via Bernstein-Zelevinsky theory (Corollary \ref{cor all same models}) and then using the reduction to basic case similar to \cite{GGP12}. The functorial connection is a key difference of our study from the one in \cite{GGP12}. We remark that we also deduce the equal rank Fourier-Jacobi case from the basic case of restricting $G_{n+1}$ to $G_n$ representations, which differs from that some results (such as multiplicity one theorems) are proved separately for equal rank Fourier-Jacobi models.

In more detail, let 
\[  H_r^R=\left\{  \begin{pmatrix} g & & x \\ & 1 & v^t \\ & & u \end{pmatrix} : g \in G_r, x \in Mat_{r \times (n-r)}, v\in F^{n-r}, u \in U_{n-r} \right\}\subset G_{n+1} ,
\]
where $U_{n-r}$ is the subgroup of unipotent upper triangular matrices. It is sometimes referred to a Rankin-Selberg subgroup. Let $\psi$ be a nondegenerate character on a subgroup $U_{n-r} \ltimes F^{n-r}$, extending trivially to $H_r^R$ (also see Section \ref{ss bessel}). We show that the restriction problem for a Bessel model or a Fourier-Jacobi model is equivalent to the problem of determining the corresponding Rankin-Selberg model (Corollary \ref{cor all same models}), i.e. determining if
\[  \mathrm{Hom}_{H_r^R}(\pi_1 \otimes \psi \otimes \nu^{-(n-r)/2}, \pi_2) \neq 0,
\]
where $\pi_1$ and $\pi_2$ are respective irreducible $G_{n+1}$ and $G_r$ representations. 

The second generalization is on Ext-branching laws. The generic case for Ext-branching law is simpler: for respective generic irreducible representations $\pi_1$ and $\pi_2$ of $G_{n+1}$ and $G_n$, 
\[  \mathrm{Hom}_{G_n}(\pi_1, \pi_2)\cong \mathbb C, \quad \mbox{ and } \quad \mathrm{Ext}_{G_n}^i(\pi_1, \pi_2)=0,  \quad \mbox{ for $i\geq 1$ }.
\]
The Ext-vanishing part is conjectured by D. Prasad \cite{Pr18} and proved in \cite{CS18}, and the Ext-result also extends to standard modules in \cite{Ch21}. One may consider an analogous problem of Ext-branching laws for Arthur representations. However, there is no such general Ext-vanishing result for Arthur representations, and we do not have a way predicting non-vanishing Ext at the moment.

Nevertheless, we formulate a conjecture in Section \ref{conj ext sum}, which reduces computations of Ext-groups for branching laws to computation of Ext-groups of derivatives. The conjecture is partly based on the derivative approach in \cite{GGP19}, as well as some examples computed in this paper.

\subsection{Outline of the proof of non-tempered GGP}

We shall consider the reformulated problem in Section \ref{ss rep the reform}. Let 
\begin{align} \label{eqn piM product express}
 \pi_M =\pi_{p,1} \times \ldots \times \pi_{p,r} \in \mathrm{Alg}(G_{n+1}), \end{align}
and
\begin{align} \label{eqn piN product express}
 \pi_N = \pi_{q,1} \times \ldots \times \pi_{q,s} \in \mathrm{Alg}(G_n),
\end{align}
where each $\pi_{p,i}$ and $\pi_{q,j}$ is an (irreducible) Speh representation.

The proof is on the induction of the total number of factors $\pi_{p,i}$ and $\pi_{q,j}$ which are not cuspidal representations. The basic case is that all factors are cuspidal representations. Then the associated Arthur parameters $M_A$ and $N_A$ are automatically relevant. Since the representations $\pi_M$ and $\pi_N$ are generic in this case, we always have $\mathrm{Hom}_{G_n}(\pi_M, \pi_N) \neq 0$.

The strategy of the general case is to find a suitable filtration on $\pi_M|_{G_n}$
\[  0 \rightarrow \lambda \rightarrow \pi_M|_{G_n} \rightarrow \omega \rightarrow 0
\]
such that 
\begin{align} \label{eqn hom ext vanishing}
  \mathrm{Hom}_{G_n}(\omega, \pi_N)=\mathrm{Ext}^1_{G_n}(\omega, \pi_N) =0 
\end{align}
and $\mathrm{Hom}_{G_n}(\lambda,\pi_N)$ can be transferred to another Hom space computable from the inductive case. Now a long exact sequence argument gives 
\[ \mathrm{Hom}_{G_n}(\pi_M|_{G_n}, \pi_N) \cong \mathrm{Hom}_{G_n}(\lambda, \pi_N) \]
and so one concludes the former from the latter one. The way to find such filtration is based on a combination of Bernstein-Zelevinsky filtration and Mackey theory, and (\ref{eqn hom ext vanishing}) would follow from comparing cuspidal supports on $\omega$ and $\pi_N$. A more systematic filtration is given in Proposition \ref{cor filtration par ind mod}.

In more detail, an Arthur type representation $\pi_M$ is written as a product of Speh representations in (\ref{eqn piM product express}). Now we write $\pi_{p,k}=v_{\rho_k}(m_k,d_k)$ for all $k$. As shown in Proposition \ref{prop bessel transfer}, there is a duality between the original restriction problem $\mathrm{Hom}_{G_n}(\pi_M, \pi_N) \neq 0$ and the dual restriction problem 
\[\mathrm{Hom}_{G_{n+1}}(\sigma' \times \pi_N, \pi_M) \neq 0 ,\] where $\sigma'$ is a certain unitarizable cuspidal representation of $G_2$. With the commutation of the Speh representations in the product, we may assume that $m_1+d_1$ is the largest among all Speh representation factors in $\pi_M$ and $\pi_N$. Such arrangement allows one (easily) finds a suitable filtration to obtain the vanishing (\ref{eqn hom ext vanishing}).

Now some cuspidal support consideration in the filtration reduces to the study of the bottom layer (of the filtration):
\begin{align} \label{eqn transfer}  \mathrm{Hom}_{G_{n+1}}( v^- \times ((\Pi \bar{\times} \pi')|_{G_k}), \pi_N) ,
\end{align}
for some $G_k$, where $v=v_{\rho_1}(m_1,d_1)$, $\pi'=\pi_{p,2}\times \ldots \times \pi_{p,r}$, $\Pi$ is the Gelfand-Graev representation and $\bar{\times}$ indicates the mirabolic induction in Section \ref{sec mir ind}. A key here is that a Gan-Gross-Prasad type reduction can be used to transfer the study of $(\Pi \bar{\times} \pi')$ to $(\sigma \times \pi')|_{G_k}$ for some suitable choice of a unitarizable cuspidal representation $\sigma$. (Here $\sigma \times \pi'$ is an irreducible $G_{k+1}$-representation.)

Now $\sigma \times \pi'$ is still Arthur type and so induction can be applied. It is clear that if $\lambda$ is a quotient of $(\sigma \times \pi')|_{G_k}$, then $v^- \times \lambda$ is still a quotient of $v^- \times ((\sigma \times \pi')|_{G_k})$, which basically deals with the if direction. The converse of the statement is not true in general, but holds under suitable assumption that fulfills our purpose. For which, we have to study the product with $\pi_{p,1}^-$ preserves extensions in some situations (Corollary \ref{cor stronger}), which handles the only if direction. We also need some product preserving irreducibility results from \cite{LM16}.

\subsection{Remarks}

For irreducible generic quotients of $G_n$ appearing in an irreducible generic representation of $G_{n+1}$ (also known as generic GGP conjecture for $\mathrm{GL}$-case), it is shown by Rankin-Selberg integrals \cite{JPSS83, Pr93}. In \cite{CS18}, G. Savin and the author give another proof for the generic case using variations of Bernstein-Zelevinsky filtrations. We also remark that in the above outline, one may replace the mirabolic induction $\Pi \bar{\times} \pi$ with certain Rankin-Selberg model discussed in Section \ref{sec general}. Such interpretation is later motivated by the approach in the generic case of orthogonal groups by M\oe glin-Waldspurger \cite{MW12}.

Our method for Arthur type representations is again a variation of Bernstein-Zelevinsky filtration method  which exploits the product structure of Arthur representations. To illustrate how the refinement gives more information, we consider respective representations in $\mathrm{GL}_5(F)$ and $\mathrm{GL}_4(F)$ in \cite[Remark 5.6]{GGP19} with $A$-parameters:
\[   M_A = 1 \otimes \mathrm{Sym}^0(\mathbb C^2) \otimes \mathrm{Sym}^2(\mathbb C^2) \oplus 1 \otimes \mathrm{Sym}^0(\mathbb C^2) \otimes \mathrm{Sym}^0(\mathbb C^2) \oplus 1 \otimes \mathrm{Sym}^0(\mathbb C^2) \otimes \mathrm{Sym}^0(\mathbb C^2) 
\]
and
\[  N_A = 1 \otimes \mathrm{Sym}^0(\mathbb C^2) \otimes \mathrm{Sym}^1(\mathbb C^2) \oplus  1 \otimes \mathrm{Sym}^1(\mathbb C^2) \otimes \mathrm{Sym}^0(\mathbb C^2) .
\]
(Here $1$ is the trivial representation of the Weil group and the first $\mathrm{Sym}^k$ factor is the irreducible $(k+1)$-dimensional representation of the $\mathrm{SL}_2(\mathbb C)$ in the Weil-Deligne group.) Their respective representations take the form:
\[   \pi_1 = \langle [\nu^{-1}, \nu] \rangle \times 1  \times 1, \mbox{ and } \quad  \pi_2 = \langle [\nu^{-1/2}, \nu^{1/2}] \rangle \times \mathrm{St}([\nu^{-1/2}, \nu^{1/2}]) .\]
(Here $1$ is the trivial character of $F^{\times}$.) Now the Mackey theory gives two layers on $\pi_1|_{G_4}$ 
\[   \langle [\nu^{-1/2},\nu^{3/2}] \rangle \times ((1 \times 1 )|_{G_1}), \mbox{ and } \quad  \langle [\nu^{-1/2}, \nu^{1/2}] \rangle \times  ((1|_{M_1} \times 1 \times 1)|_{G_2})\]
Set $\tau =\langle [\nu^{-1/2}, \nu^{1/2}] \rangle$. A key difference of our method from the one in \cite{GGP19} is to use transfer in (\ref{eqn transfer}) to deduce that $\tau \times  ((1|_{M_1} \times 1 \times 1)|_{G_2})$ has a quotient of $\pi_2$, as $G_4$ representations. The above filtration is coarser than the full Bernstein-Zelevinsky filtration, but has the advantage of using transfer and induction as mentioned before. It can also deal with an obstruction in \cite[Remark 5.6]{GGP19}.


An irreducible cuspidal representation of $G_n$ restricted to the mirabolic subgroup is isomorphic to the Gelfand-Graev representation, which is an essential step in our proof. This classical fact plays crucial roles, and is generalized to essentially square-integrable representations when restricted to $G_{n-1}$ via Hecke algebra realization \cite{CS18, CS19} (also see \cite{Ch19} for further generalization to representations restricted to be projective), but we do not critically need any Hecke algebra technique in this paper. We also remark that such fact also plays important roles, for example, in the reductions in \cite{GGP12} and in proving the Ext-vanishing theorem in \cite{CS18}. 




Our approach in only if direction only requires a study of producting with one Speh representations (rather than several generalized Speh representations), compared with the study of multiple products of generalized Speh representations in \cite{Gu18}. We also work directly in the $p$-adic group without passing to other categories, and so the method is different from \cite{Gu18}. We show under some conditions on cuspidal supports that producting with a Speh representation preserves extensions and is a fully-faithful functor. This improves one of results of Lapid-M\'inguez \cite{LM16} which shows producting with Speh representations preserves irreducibility under a related condition. 


\subsection{Summary of results and structure of the paper}

We summarize the key results of this paper below:
\begin{enumerate}
\item A proof of Conjecture \ref{conj ggp orig} in the case of a non-Archimedean field (Theorem \ref{thm conj ggp})
\item Generalize Conjecture \ref{conj ggp orig} to Bessel, Fourier-Jacobi and other mixed models (Theorem \ref{thm general cases})
\item A filtration as a tool to study restriction problem for parabolically induced modules (Proposition \ref{cor filtration par ind mod})
\item Product with a Speh representation preserves indecomposability and is a fully-faith functor under some conditions (Theorems \ref{thm preserve extension} and \ref{thm fully faith product})
\end{enumerate}

In Section \ref{sec notation prelim}, we set up notations and recall some results such as properties of Speh representations. In Section \ref{sec mir ind}, we study parabolically induced modules restricted to mirabolic subgroups. In Section \ref{sec proof conj}, we prove Conjecture \ref{conj ggp orig} for non-Archimedean case. In Section \ref{sec general}, we generalize results to general cases including Bessel models and Fourier-Jacobi models. In Section \ref{sec fj bz theory}, we establish connections between models. In Section \ref{ss ext branching}, we study Ext-branching laws for Arthur type representations. In Sections \ref{ss prod preserve} and \ref{s product functor}, we prove Theorems \ref{thm preserve extension} and \ref{thm fully faith product}.

\subsection{Acknowledgement}
This project grows out from discussions with Dipendra Prasad, and the author would like to thank him for helpful discussions and comments. He would also like to thank Gordan Savin for discussions on various topics and helpful comments. The author would also like to thank Max Gurevich for helpful correspondences on the preprint. The author would like to thank the referee for careful reading and useful comments.

\section{Notations and Preliminaries} \label{sec notation prelim}

\subsection{Bernstein-Zelevinsky functors} \label{ss bz functor}


For a connected reductive group $G$, let $\mathrm{Alg}(G)$ be the category of smooth (complex) representations of $G$. Let $G_n=\mathrm{GL}_n(F)$. All representations in this paper are smooth and we usually drop the term 'smooth'. For a representation $\pi$ of $G_n$, set $n_{\pi}=n$.

Let $G=G_n$. For a closed subgroup $H$ of $G$ and a representation $\pi$ in $\mathrm{Alg}(H)$, let $\mathrm{Ind}_{H}^{G} \pi$ to be the space of smooth functions $f: G \rightarrow \pi$ satisfying $f(hg)=\delta(h)^{1/2} h.f(g)$, where $\delta$ is the modulus function of $H$. The $G$-action on $\mathrm{Ind}_H^G\pi$ is given by $(g.f)(g_0)=f(g_0g)$ for any $g, g_0 \in G$. Let $\mathrm{ind}_H^G \pi$ be the subrepresentation of $\mathrm{Ind}_H^G\pi$ containing all functions with compact support modulo $H$. We shall use ${}^u\mathrm{ind}$ and ${}^u\mathrm{Ind}$ for corresponding unnormalized inductions of $\mathrm{ind}$ and $\mathrm{Ind}$ respectively. Those functors $\mathrm{Ind}, \mathrm{ind}, {}^u\mathrm{Ind}, {}^u\mathrm{Ind}$ are exact \cite[Proposition 2.25(a)]{BZ76}.

Let $M_n$ be the mirabolic subgroup of $G_n$ i.e. $M_n$ is the subgroup of $G_n$ with all the matrices with the last row $(0, \ldots, 0,1)$. We shall also regard $G_{n-1}$ as a subgroup of $M_n$ via the embedding $g \mapsto \begin{pmatrix} g & \\ & 1 \end{pmatrix}$. Thus we have a chain of subgroups:
\[  1=G_0= M_1\subset \ldots \subset         M_{n-1} \subset  G_{n-1} \subset M_n \subset  G_n.
\]
For $\pi \in \mathrm{Alg}(G_n)$, we may simply write $\pi|_{M}$ for the restriction $\pi|_{M_n}$


Let $V=V_{n-1}$ be the unipotent radical of $M_n$. Let $\bar{\psi}: F \rightarrow \mathbb{C}$ be a non-degenerate character. Let $\psi: V \rightarrow \mathbb{C}$ by $\psi(v)=\bar{\psi}(v_{n-1})$, where $v_{n-1}$ is the last entry in $v$. Note the action of $M_{n-1}$ stabilizes $\psi: V \rightarrow \mathbb{C}$. For a character $\lambda$ of $V$ and a representation $\pi$ of $M_n$, define 
\[  \pi_{V, \lambda} = \delta^{-1/2} \pi /\langle v.x-\lambda(v)x : v \in V, x\in \pi \rangle ,
\]
where $\delta$ is the modulus function of $M_n$. When $\lambda=1$ (resp. $\lambda=\psi$), we regard $\pi_{V,\lambda}$ as $G_{n-1}$-representation (resp. $M_{n-1}$-representation).

Define $\theta=\theta_n: G_n \rightarrow G_n$ by $\theta(g)=g^{-t}$, the Gelfand-Kazhdan involution \cite[Section 7]{BZ76}. For any irreducible representation $\pi$ of $G_n$, $\theta(\pi)\cong \pi^{\vee}$.

Define 
\[  \Phi^+: \mathrm{Alg}(M_n) \rightarrow \mathrm{Alg}(M_{n+1}) ; \quad \Psi^+:\mathrm{Alg}(G_n) \rightarrow \mathrm{Alg}(M_{n+1}) 
\]
\[  \Phi^-: \mathrm{Alg}(M_{n+1})\rightarrow \mathrm{Alg}(M_n); \quad \Psi^-: \mathrm{Alg}(M_{n+1}) \rightarrow \mathrm{Alg}(G_n) .
\]
by
\[  \Phi^+(\pi)=\mathrm{ind}_{M_nV_n}^{M_{n+1}} \pi\boxtimes \psi, \quad \Psi^+(\pi)=\mathrm{ind}_{G_nV_n}^{M_{n+1}} \pi \boxtimes 1 ,
\]
\[  \Phi^-(\pi)=\pi_{V_n, \psi}, \quad \Psi^-(\pi)=\pi_{V_n, 1} .
\]

In particular, $\Psi^+$ is just an inflation of representations. Some major properties of the functors \cite[Proposition 3.2]{BZ77}:
\begin{enumerate}
\item All the above functors are exact. 
\item $\Phi^-$ is left-adjoint to $\Phi^+$ and $\Psi^-$ is left-adjoint to $\Phi^+$. 
\item $\Phi^-\Psi^+=0$ and $\Psi^- \Phi^+=0$
\item There is an exact sequence: 
\[ 0 \rightarrow \Phi^+\Phi^- \rightarrow \mathrm{Id} \rightarrow \Psi^+\Psi^- \rightarrow 0  \]
\item All the irreducible representations of $M_n$ are isomorphic to $(\Phi^+)^{k-1}\Psi^+(\pi)$ for some $k$ and some irreducible smooth $G_{n-k}$-representation. 
\item\cite[5.18]{BZ76} For any cuspidal representation $\sigma$ of $G_n$, $\sigma|_{M_n} \cong (\Phi^+)^{n-1}(1)$. Here $1$ is the $1$-dimensional representation of $M_1$.
\end{enumerate}


Denote, the Gelfand-Graev representation,
\begin{align}
 \Pi_n :=(\Phi^+)^{n-1}(1) \in \mathrm{Alg}(M_n) .
\end{align}

Let $\nu=\nu_n: G_n\rightarrow \mathbb{C}$ be a character given by $\nu(g)=|\mathrm{det}(g)|_F$. For $\pi \in \mathrm{Alg}(G_n)$, the $k$-th right and left derivatives of $\pi$ are respectively defined as:
\[  \pi^{(k)} =\Psi^-(\Phi^-)^{k-1}(\pi|_{M_n}) , \quad {}^{(k)}\pi=\theta(\theta(\pi)^{(k)}) .
\]
and the $k$-th shifted right and left derivatives of $\pi$ are defined as:
\[  \pi^{[k]}=\nu^{1/2}\cdot \pi^{(k)} , \quad {}^{[k]}\pi=\nu^{-1/2} \cdot {}^{(k)}\pi .
\]

Let $k^*$ be the largest integer such that $\pi^{(k^*)} \neq 0$. We shall call $\pi^{(k^*)}$ to be the highest derivative of $\pi$, and $k^*$ to be the level of $\pi$. We also set $\pi^-=\pi^{[k^*]}$.

\subsection{Parabolic induction and Jacquet functors} \label{ss para induction}

Let $U_n$ be the subgroup of $G_n$ containing all unipotent upper triangular matrices. Let $N_i$ be the unipotent subgroup of $G_n$ containing matrices of the form $\begin{pmatrix} I_{n-i} & u \\ & I_i \end{pmatrix}$ for any $(n-i)\times i$ matrices $u$ over $F$. We regard $G_{n-i}\times G_i$ as a subgroup of $G_n$ via the embedding $(g_1,g_2)\mapsto \mathrm{diag}(g_1,g_2)$. Let $P_i$ be the parabolic subgroup $(G_{n-i}\times G_i )N_i$.

 For $\pi_1 \in \mathrm{Alg}(G_{n-i})$ and $\pi_2 \in \mathrm{Alg}(G_i)$, define the product of $\pi_1$ and $\pi_2$ as
\[  \pi_1 \times \pi_2 = \mathrm{Ind}_{(G_{n-i}\times G_i)\ltimes N_i}^{G_n} \pi_1 \boxtimes \pi_2 \boxtimes 1 .
\]
For a family of representations $\pi_i \in \mathrm{Alg}(G_{n_i})$ ($i=1,\ldots, k$), define
\[  \pi_1 \times \ldots \times \pi_k:=\pi_1\times (\ldots \times (\pi_{k-1}\times \pi_k)\ldots ).
\]
The parabolic induction is an exact functor \cite{BZ76}. For more properties for parabolic inductions, see \cite{LM16}.

Let $N_i^-=N_i^t$ be the opposite unipotent subgroup. For $\pi \in \mathrm{Alg}(G_n)$, we shall denote by $\pi_{N_i}$ and $\pi_{N_i^-}$ be the corresponding normalized Jacquet modules, as $G_{n-i}\times G_i$-representations. They are also exact functors. Since the parabolic induction has usual and opposite Jacquet functors as left and right adjoint functors respectively, parabolic induction also preserves injective and projective objects. 

For an irreducible representation $\pi$ of $G_k$, there is a unique set (with multiplicities) of cuspidal representations $\rho_1, \ldots, \rho_r$ such that $\pi$ is a composition factor of $\rho_1 \times \ldots \times \rho_r$, and we denote the multiset $\mathrm{cupp}(\pi)=\left\{ \rho_1, \ldots, \rho_r \right\}$, and denote the set $\mathrm{cupp}_{\mathbb{Z}}(\pi)=\left\{ \nu^i \rho_j \right\}_{i\in \mathbb Z, j=1,\ldots, r}$.

\subsection{Bernstein-Zelevinsky filtrations} \label{ss BZ filtration}

Since Bernstein-Zelevinsky filtration \cite[Section 3.5]{BZ77} (and its variations) is a main tool in this article, we recall in this section. 

Let $\pi$ be in $\mathrm{Alg}(G_{n+1})$. Then, $\pi|_{G_n}$ admits a filtration:
\[ 0= \pi_{n+1} \subset \pi_n \subset \ldots \pi_1 \subset \pi_0 = \pi|_{G_n} 
\]
such that
\[   \pi_{k-1}/\pi_{k} \cong  \pi^{[k]} \times \mathrm{ind}_{U_{k-1}}^{G_{k-1}} \psi_k ,
\]
where $\psi_k$ is a non-degenerate character on $U_{k-1}$. Note that $\mathrm{ind}_{U_{k-1}}^{G_{k-1}}\psi_k \cong \Pi_{k+1}|_{G_k}$. There is a 'left' version of the filtration (see \cite{CS18} and \cite{Ch21}), while we do not need this in this article.

\subsection{Speh representations and Zelevinsky segments} \label{sec zel segments}

Let $\rho$ be an irreducible cuspidal representation of $G_m$. For any $a,b \in \mathbb{C}$ with $b-a \in \mathbb{Z}_{\geq 0}$, a Zelevinsky segment $\Delta=[\nu^a\rho, \nu^b\rho]$ is the set $\left\{ \nu^a\rho, \nu^{a+1}\rho, \ldots , \nu^b\rho \right\}$, and we denote $a(\Delta)=\nu^a \rho$ and $b(\Delta)=\nu^b \rho$. Denote by $\langle \Delta \rangle$ (resp. $\mathrm{St}(\Delta)$) the unique submodule (resp. quotient) of $\nu^a\rho \times \ldots \times \nu^b\rho$.

 A Zelevinsky multisegment is a multiset of Zelevinsky segments. For a Zelevinsky multisegment $\mathfrak m=\left\{ \Delta_1, \ldots , \Delta_r \right\}$, denote by $\langle \mathfrak m \rangle$ the unique irreducible subrepresentation  of $\langle \Delta_1 \rangle \times \ldots \times \langle \Delta_r \rangle$, and denote by $\mathrm{St}(\mathfrak m)$ the unique irreducible quotient of $\mathrm{St}(\Delta_1)\times \ldots \times \mathrm{St}(\Delta_r)$,  where $\Delta_1, \ldots, \Delta_r$ are ordered in the way as in \cite[Theorem 6.1]{Ze80}.  We also denote the parabolic induction $\langle \Delta_1 \rangle \times \ldots \times \langle \Delta_r \rangle$ by $\zeta(\mathfrak m)$.

Let $\mathrm{Irr}^c(G_k)$ be the set of all (isomorphism classes of) irreducible cuspidal representations of $G_k$, and let $\mathrm{Irr}^c=\sqcup_{k \geq 0}  \mathrm{Irr}^c(G_k)$. Let $\mathrm{Irr}^{u,c}(G_k)$ be the set of irreducible unitarizable cuspidal representations of $G_k$. Let $\mathrm{Irr}^{u,c} =\sqcup_{k\geq 0}\mathrm{Irr}^{u,c}(G_k)$.

Let $\rho \in \mathrm{Irr}^{c}(G_m)$. For a positive integer $d$, define 
\[  \Delta_{\rho}( d)=[\nu^{-(d-1)/2}\rho, \nu^{(d-1)/2}\rho ] .
\]
For a positive integer $m$, define
\[  u_{\rho}(m,d) = \langle \left\{ \nu^{-(m-1)/2}\Delta_{\rho}(d), \ldots, \nu^{(m-1)/2}\Delta_{\rho}(d) \right\} \rangle .
\]
When $\rho$ is unitarizable, we shall call those representations to be Speh representations, and they are unitarizable \cite[Section 8]{Be84} (see \cite{Ta86}). 

In Section \ref{ss rep the reform}, we also introduce the notion $v_{\rho}(m,d)$. The two notions coincide (and here we do not assume $\rho$ to be unitarizable):

\begin{lemma} \cite[Theorem A10]{Ta86} \label{lem st ze transfer} For any $\rho \in \mathrm{Irr}^{c}(G_n)$, any $d,m \geq 1$, 
\[  v_{\rho}(m,d) \cong u_{\rho}(m,d) .
\]
\end{lemma}

The above result can also be deduced from M\oe glin-Waldspurger algorithm.

Explicit derivatives of a Speh representation are particularly simple to describe, and one refers to \cite{LM14} (also see \cite[Section 7]{CS19}). We collect some useful information for our study:

\begin{lemma} \cite[Theorem 14]{LM14} \label{lem speh restrict}
Let $\pi=u_{\rho}(m,d)$ be a Speh representation.
\begin{enumerate}
\item The level of $\pi$ is $n_{\rho}m$. 
\item If $k$ is not the level of $\pi$ and $\pi^{[k]}\neq 0$, then the cuspidal support of $\pi^{[k]}$ contains $\nu^{(d+m-2)/2+1/2}\rho$. 
\item If $k$ is the level of $\pi$, then $\pi^-= \pi^{[k]} \cong u_{\rho}(m,d-1)$ and $\pi^{(k)} \cong \nu^{-1/2}u_{\rho}(m,d-1)$. 
\end{enumerate}
\end{lemma}

\subsection{Weakly relevant condition for general branching law}

For a segment $\Delta$, set ${}^{[0]}\Delta=\nu^{-1/2}\Delta$ and $\Delta^{[0]}=\nu^{1/2}\Delta$. For a segment $\Delta=[\nu^a\rho, \nu^b\rho]$, set $\Delta^-=[\nu^a\rho, \nu^{b-1}\rho]$ and ${}^-\Delta=[\nu^{a+1}\rho, \nu^b\rho]$; and set $\Delta^{[-]}=\nu^{1/2}\Delta^-$ and ${}^{[-]}\Delta=\nu^{-1/2}\cdot {}^-\Delta$. 

 Let $\mathfrak m$ and $\mathfrak n$ be two Zelevinsky multisegments. We say that $\mathfrak m$ and $\mathfrak n$ are {\it weakly relevant} if there exists segments 
\[ \Delta_{p,1}, \ldots , \Delta_{p,r}, \quad \Delta_{q,1}, \ldots, \Delta_{q,s} \]
and
\[  \Delta_{a,1},\ldots, \Delta_{a,k}, \quad \Delta_{b,1},\ldots, \Delta_{b,l} \] such that 
\[  \mathfrak m = \left\{ \Delta_{p,1}, \ldots, \Delta_{p,r}, \Delta_{q,1}^{[-]}, \ldots, \Delta_{q,s}^{[-]}\right\} \cup \left\{ \Delta_{a,1},\ldots, \Delta_{a,k}, {}^{[0]}\Delta_{b,1}, \ldots {}^{[0]}\Delta_{b,l} \right\},\]
and
\[ \mathfrak n = \left\{ {}^{[-]}\Delta_{p,1}, \ldots, {}^{[-]}\Delta_{p,r}, \Delta_{q,1}, \ldots, \Delta_{q,s} \right\} \cup \left\{ {}^{[0]}\Delta_{a,1},\ldots, {}^{[0]}\Delta_{a,k}, \Delta_{b,1},\ldots, \Delta_{b,l} \right\}.\]

While we do not need the following result, it gives one guiding principle in general smooth branching law. One may also compare with an Archimedean result \cite{GS20} in terms of wavefront sets. We remark that the converse is not true in general (for example, see the quotient branching law for the Steinberg representation in \cite{CS18}).

\begin{proposition}
Let $\pi_1, \pi_2$ be irreducible smooth representations of $G_{n+1}$ and $G_n$ respectively. Let $\mathfrak m$ and $\mathfrak n$ be their associated Zelevinsky multisegments. If $\mathrm{Hom}_{G_n}(\pi_1, \pi_2) \neq 0$, then $\mathfrak m$ and $\mathfrak n$ are weakly relevant.
\end{proposition}

\begin{proof}
Since $\mathrm{Hom}_{G_n}(\pi_1, \pi_2) \neq 0$, we have that $\mathrm{Hom}_{G_n}(\theta(\zeta(\mathfrak m))^{\vee}, \zeta(\mathfrak n)) \neq 0$. (We remark that $\theta(\zeta(\mathfrak m))^{\vee}$ has a quotient of $\langle \mathfrak m \rangle$ as $\theta(\langle \mathfrak m \rangle)^{\vee}\cong \langle \mathfrak m \rangle$.) Let $\zeta_1=\theta(\zeta(\mathfrak m))^{\vee}$ and $\zeta_2=\zeta(\mathfrak n)$.  Now using Bernstein-Zelevinsky filtration (Section \ref{ss BZ filtration}, also see \cite[Lemma 2.4]{CS18} and \cite[Lemma 2.1]{Ch19}), we obtain that, for some $i$,
\[  \mathrm{Hom}_{G_{n+1-i}}(\zeta_1^{[i+1]}, {}^{(i)}\zeta_2) \neq 0 .
\]
Write $\mathfrak m=\left\{ \Delta_1, \ldots, \Delta_k \right\}$ and $\mathfrak n=\left\{ \widetilde{\Delta}_1, \ldots, \widetilde{\Delta}_l \right\}$. For convenience, we also set the notion $\Delta^0=\Delta$ and ${}^0\Delta=\Delta$ (i.e. no effect on $\Delta$). Using geometric lemma on $\zeta_1$ (with suitable arrangement of segments, see for example the proof of \cite[Lemma 6.3]{Ch19}), we obtain a filtration on $\zeta_1^{[i+1]}$ whose successive quotients are $\theta(\zeta(\mathfrak p)^{\vee})$ with $\mathfrak p$ taking the form:
\[ \mathfrak p = \left\{ \nu^{1/2}\Delta_1^{\#}, \ldots,\nu^{1/2} \Delta_k^{\#} \right\}
\]
and similarly a filtration on ${}^{(i)}\zeta_2$ whose successive quotients are $\zeta(\mathfrak q)$ with $\mathfrak q$ taking the form:
\[ \mathfrak q = \left\{ {}^{\#}\widetilde{\Delta}_1, \ldots, {}^{\#}\widetilde{\Delta}_l \right\},
\]
where each $\#=-$ or $=0$. Now the previous non-vanishing Hom implies that 
\[  \mathrm{Hom}_{G_n}(\theta(\zeta(\mathfrak p)^{\vee}), \zeta(\mathfrak q)) \neq 0 
\]
for some $\mathfrak p$ and $\mathfrak q$ of the form. The ordering \cite[Theorem 7.1]{Ze80} of Zelevinsky classification implies that a non-zero map in $\mathrm{Hom}_{G_n}(\theta(\zeta(\mathfrak p)^{\vee}), \zeta(\mathfrak q))$ factors through the Zelevinsky submodule of $\zeta(\mathfrak q)$ and hence $\mathfrak p=\mathfrak q$ for some $\mathfrak p$ and $\mathfrak q$ taking the above form. In other words, we have that for each $j$, $\nu^{1/2}\Delta_j^{-} = \widetilde{\Delta}_{i_j}$ or $\nu^{1/2}\Delta_j^-=1$ (as $G_0$-representation), or $\Delta_j=\nu^{-1/2}\cdot {}^-\widetilde{\Delta}_{i_j}$, or $\nu^{1/2}\Delta_j=\widetilde{\Delta}_{i_j}$, or $\nu^{1/2}\Delta_j^-={}^-\widetilde{\Delta}_{i_j}$ for some $i_j$. These conditions give the weakly relevant condition on the pair $(\mathfrak m, \mathfrak n)$.

\end{proof}

\subsection{Ext-vanishing on cuspidal supports}

The following result is standard. Since we shall frequently use the following result, we give a proof on it.

\begin{lemma} \label{lem ext vanishing}
Let $\pi_1$ be an irreducible $G_{n-i}$-representation and let $\pi_2 \in \mathrm{Alg}(G_i)$ (not necessarily admissible). Let $\pi$ be an admissible $G_n$-representation. Suppose that for any simple composition factor $\tau$ of $\pi$, $\mathrm{cupp}(\tau) \cap \mathrm{cupp}(\pi_1) =\emptyset$. Then, for all $i$,
\[  \mathrm{Ext}^i_{G_n}(\pi_1\times \pi_2, \pi)=0 \]
\end{lemma}

\begin{proof}
It suffices to prove for $\pi$ to be irreducible. One first applies Frobenius reciprocity, for any $i$,
\[  \mathrm{Ext}^i_{G_n}(\pi_1 \times \pi_2, \pi) \cong \mathrm{Ext}^i_{G_{n-i}\times G_i}(\pi_1\boxtimes \pi_2, \pi_{N_i^-}) .\]
Since $\pi_{N_i^-}$ is admissible, it suffices to check the Ext-vanishing for each simple composition factor $\tau'$ of $\pi_{N_i^-}$. Now we write $\tau'=\tau_a \boxtimes \tau_b$ for simple $G_{n-i}$ and $G_i$ representations $\tau_a$ and $\tau_b$ respectively. Now we have that $\mathrm{Ext}^i_{G_i}(\pi_1, \tau_a)=0$ since, by using the cuspidal support condition, one can find an element in the Bernstein center which acts by a different scalar on $\pi_1$ and $\tau_a$. Now we conclude $\mathrm{Ext}^i_{G_n}(\pi_1 \times \pi_2, \tau')=0$ by K\"unneth formula.
\end{proof}





\section{Mirabolic Induction} \label{sec mir ind}

In this section, we discuss inductions involving mirabolic subgroups, which will be used in Sections \ref{sec proof conj}, \ref{sec general} and \ref{sec fj bz theory}.

\subsection{Mirabolic induction} \label{ss mir ind}

Let $\tau \in \mathrm{Alg}(M_m)$ and let $\pi \in \mathrm{Alg}(G_n)$. Define two types of mirabolic inductions, similar to \cite[4.12]{BZ77}.

\begin{enumerate}
\item  Type 1: Let $Q = P_m \cap M_{n+m} \subset G_{n+m}$ i.e. 
\[  Q = \left\{ \begin{pmatrix} g & u \\ & m \end{pmatrix} :  g \in G_n, m \in M_m, u \in Mat_{n\times m}       \right\} .
\]
Let $\epsilon: Q\rightarrow \mathbb{C}$ be the identity. 
\item Type 2: Let $Q= P_m^t \cap M_{n+m} \subset G_{n+m}$ i.e. 
\[  Q= \left\{ \begin{pmatrix} g & & \\ u & h & v \\  & & 1 \end{pmatrix}  :  g \in G_n, u \in Mat_{m-1, n}, h \in G_{m-1}, v \in F^{m-1}   \right\} .
\]
Let $\epsilon: Q: \rightarrow \mathbb{C}$ given by $\epsilon=\nu^{-1/2}$. 
\end{enumerate}

For type 1 (resp. type 2), extend $\pi \boxtimes \tau$ trivially to $Q$. Define the $M_{n+m}$-representation $\pi \bar{\times} \tau$ (resp. $\tau \bar{\times} \pi$) to be the space of smooth functions $f: M_{n+m} \rightarrow \pi \boxtimes \tau$ satisfying $f(qg)=\epsilon(q)\delta(q)^{1/2} q.f(g)$ for any $q \in Q$ and $g \in M_{n+m}$, and $f$ is compactly-supported modulo $Q$, where $\delta$ is the modulus function of $Q$. 

In type 1, when restricting to $G_{n+m-1}$, we have 
\begin{align} \label{eqn restriction to g first factor}
 (\pi \bar{\times} \tau)|_{G_{n+m-1}} \cong (\nu^{1/2}\pi) \times (\tau|_{G_{m-1}} ),
\end{align}
where the isomorphism is given by $f \mapsto (g\mapsto f(\mathrm{diag}(g,1)))$. Here we naturally identify $\pi \boxtimes \tau$ and $(\nu^{1/2}\pi)\boxtimes (\tau|_{G_{m-1}})$. We may also sometimes simply write $\times$ for $\bar{\times}$.

\subsection{Associative property}
The following lemma follows from an inspection. We omit the details.

\begin{lemma} \label{lem associative} \cite{Ze80}
 Let $\pi_1 \in \mathrm{Alg}(G_{n_1})$. Let $\pi_2 \in \mathrm{Alg}(G_{n_2})$. Let $\tau \in \mathrm{Alg}(M_r)$. Then 
\begin{enumerate}
\item $(\pi_1 \bar{\times} \tau) \bar{\times} \pi_2 \cong \pi_1 \bar{\times} (\tau \bar{\times} \pi_2)$; and
\item $(\pi_1 \times \pi_2) \bar{\times} \tau \cong \pi_1 \bar{\times} (\pi_2 \bar{\times} \tau)$;
\item $(\tau \bar{\times} \pi_1) \bar{\times} \pi_2 \cong \tau \bar{\times} (\pi_1 \times \pi_2)$.
\end{enumerate}
\end{lemma}

\subsection{From parabolic to mirabolic induction}

The appearance of mirabolic inductions comes from the study of parabolic inductions when restricting to the mirabolic subgroup via Mackey theory. The following lemma will be used several times.

\begin{lemma} \cite[Proposition 4.13]{BZ77} \label{lem short exact seq}
Let $\pi_1$ and $\pi_2$ be $G_{n_1}$ and $G_{n_2}$-representations. Then $(\pi_1 \times \pi_2)|_M$ admits a short exact sequence:
\[ 0 \rightarrow \pi_1|_M \bar{\times} \pi_2 \rightarrow  (\pi_1 \times \pi_2)|_M \rightarrow \pi_1 \bar{\times} (\pi_2|_M) \rightarrow 0
\]
\end{lemma}

\subsection{Connection to Bernstein-Zelevinsky functors}

\begin{lemma} \cite[Proposition 4.13]{BZ77} \label{lem mirabolic decomp}
Let $\pi \in \mathrm{Alg}(G_{n})$. Let $\tau \in \mathrm{Alg}(M_k)$. Then 
\begin{enumerate}
\item $\Psi^-(\tau \bar{\times} \pi) \cong \Psi^-(\tau) \times \pi$
\item $ 0 \rightarrow \Phi^-(\tau) \bar{\times} \pi \rightarrow \Phi^-(\tau \bar{\times} \pi)\rightarrow \Psi^-(\tau) \bar{\times} (\pi|_M) \rightarrow 0$ 
\end{enumerate}
\end{lemma}

The following result is standard. We omit the details.

\begin{lemma} \label{lem standard transfer to product}
For $\pi \in \mathrm{Alg}(G_r)$,
\[  (\Phi^+)^k\Psi^+(\pi) \cong \pi \bar{\times} \Pi_{k+1}
\]

\end{lemma}



It is also convenient to define another functor:
\[  \Lambda : \mathrm{Alg}(G_n) \rightarrow \mathrm{Alg}(M_{n+1}) 
\]
by 
\[\Lambda(\pi) = {}^u\mathrm{Ind}_{G_n}^{M_{n+1}}\nu^{-1/2} \pi .\]
 By definitions, $\Lambda(\pi) \cong 1|_{M_1}\bar{\times} \pi$. When $n=0$, $\Lambda$ defines an isomorphism between vector spaces.

\begin{proposition} \label{prop nontrivial bz}
Let $r\geq 0$. Let $\pi \in \mathrm{Alg}(G_r)$. For $s \geq 0$,
\[  \Pi_{s+1} \bar{\times} \pi \cong (\Phi^+)^s(\Lambda(\pi)) .
\]
\end{proposition}

\begin{proof}
Recall that $\Pi_{s+1}=(\Phi^+)^{s}(1)$ (and $1$ is the trivial representation of $M_1$) and so, by $\Psi^-\circ \Phi^+=0$, $\Psi^-(\Pi_{s+1-k})=0$ for $k=0,\ldots, s-1$. This with Lemma \ref{lem mirabolic decomp}(2) implies, for $k=0,\ldots, s-1$,  
\[ (*)\quad  \Phi^-(\Pi_{s+1-k} \bar{\times} \pi) \cong \Phi^-(\Pi_{s+1-k})\bar{\times} \pi \cong \Pi_{s-k} \bar{\times} \pi  .
\]
Here in the last isomorphism, we use $\Pi_{s+1-k}=(\Phi^+)^{s-k}(1)$ for any $k$ and $\Phi^-\circ \Phi^+\cong \mathrm{Id}$.

Now, for $0 \leq k \leq s-1$, 
\[ (**)\quad  (\Psi^-)(\Pi_{s+1-k}\bar{\times} \pi) =0, \]
where the equality follows from Lemma \ref{lem mirabolic decomp}(1) and above discussions.

Now repeatedly using Bernstein-Zelevinsky theory \cite[Proposition 3.2]{BZ77} (see property (4) of the functors in Section \ref{ss bz functor}) on $\Pi_{s+1-k}\bar{\times} \pi$ ($k=0,1,\ldots,s-1$) with (*) and (**), we have 
\[  \Pi_{s+1}\bar{\times} \pi \cong \Phi^+(\Pi_{s}\bar{\times}\pi) \cong \ldots \cong (\Phi^+)^s(\Pi_1\bar{\times} \pi) .
\]
The last isomorphism simply yeilds:
\[ \Pi_{s+1} \bar{\times} \pi \cong (\Phi^+)^s(\Lambda(\pi)) .\]
\end{proof}


\subsection{A transfer lemma}

We shall need the following transfer or reduction:

\begin{lemma} \label{lem cuspidal supp}
Let $\pi_1 \in \mathrm{Alg}(G_k)$ and $\pi_2 \in \mathrm{Alg}(G_l)$. Let $\pi_3 \in \mathrm{Alg}(G_n)$ with $n \geq l+k$. Let $a=n+1-(k+l)$. Then, for any  $\sigma$ in $\mathrm{Irr}^c(G_{n+1-(k+l)})$ such that $\sigma \notin \mathrm{csupp}_{\mathbb Z}(\nu^{-1/2}\pi_3)$, and for any $i$,
\[ \mathrm{Ext}^i_{G_n}(\pi_1\times( (\sigma \times \pi_2)|_{G_{n-k}}), \pi_3) \cong \mathrm{Ext}^i_{G_n}(\pi_1 \times ((\Pi_a \bar{\times} \pi_2)|_{G_{n-k}}), \pi_3) .
\]
  
\end{lemma}

\begin{proof}
Again Lemma \ref{lem short exact seq} gives a filtration on $(\sigma \times \pi_2)|_{M_{n+1-k}}$ as:
\begin{align} \label{eqn ses}
  0 \rightarrow  \sigma|_{M} \bar{\times} \pi_2 \rightarrow (\sigma \times \pi_2)|_{M} \rightarrow \sigma \bar{\times} (\pi_2|_{M}) \rightarrow 0 .
\end{align}
Restricting to $G_{n-k}$, this gives the filtration:
\[   0 \rightarrow  (\sigma|_{M} \bar{\times} \pi_2)|_{G_{n-k}} \rightarrow (\sigma \times \pi_2)|_{G_{n-k}} \rightarrow (\nu^{1/2}\sigma) \times (\pi_2|_{G_{l-1}}) \rightarrow 0 .
\]
With $\Pi_a=\sigma|_M$, producting with $\pi_1$ gives the exact sequence:
\begin{align} \label{eqn long exact three product}
 0 \rightarrow \pi_1 \times (( \Pi_a \bar{\times} \pi_2 )|_{G_{n-k}}) \rightarrow \pi_1  \times((\sigma \times \pi_2)|_{G_{n-k}}) \rightarrow \pi_1 \times (\nu^{1/2}\sigma) \times (\pi_2|_{G_{l-1}}) \rightarrow 0 .
\end{align}

The standard argument using second adjointness of Frobenius reciprocity and comparing cuspidal support at $\nu^{1/2}\sigma$ gives that, for all $i$, 
\[  \mathrm{Ext}^i_{G_{n}}(\pi_1 \times (\nu^{1/2}\sigma) \times (\pi_2|_{G_{l-1}}), \pi_3)= 0 .
\]
Thus long exact sequence from (\ref{eqn long exact three product}) gives that, for all $i$, 
\[  \mathrm{Ext}^i_{G_{n}}(\pi_1 \times ((\Pi_a \bar{\times} \pi_2)|_{G_{n-k}}), \pi_3) \cong \mathrm{Ext}^i_{G_{n}}(\pi_1 \times ((\sigma \times \pi_2)|_{G_{n-k}}) , \pi_3) .
\]  
\end{proof}

\subsection{A lemma on Speh representation}

\begin{lemma} \label{lem speh rep filtration}
Let $\pi=u_{\rho}(m,d)$ be a Speh representation.  Let $\pi'$ be in $\mathrm{Alg}(G_k)$. Let $n+1=n_{\pi}+k$. Let $\pi''$ be an irreducible representation of $G_{n}$ such that $\nu^{1/2}(\nu^{(m+d-2)/2}\rho)$ is not in $\mathrm{cupp}(\pi'')$. Then there exists a short exact sequence, as $G_{n}$-representations: 
\[ 0 \rightarrow K \rightarrow (\pi|_M \bar{\times} \pi')|_{G_{n}} \rightarrow Q \rightarrow 0 
\]
such that, for all $i$,
\[  \mathrm{Ext}^i_{G_{n}}(Q, \pi'') = 0,
\]
and, 
\[ K \cong  ((\nu^{-1/2}u_{\rho}(m,d-1)) \bar{\times} (\Pi_{p} \bar{\times} \pi'))|_{G_n} \cong u_{\rho}(m,d-1)\times ((\Pi_p  \bar{\times} \pi')|_{G_{k+p-1}})\] 
where  $p=n_{\rho}m$, and 
\[ \mathrm{Ext}^i_{G_{n}}(K, \pi'') \cong \mathrm{Ext}^i_{G_{n}}((\pi|_M \bar{\times} \pi')|_{G_{n}}, \pi'') .
\]
\end{lemma}

\begin{proof}

From the bottom piece of Bernstein-Zelevinsky filtration (Lemma \ref{lem speh restrict}), $\pi|_M$ has the submodule (see Section \ref{ss bz functor} and Lemma \ref{lem standard transfer to product})
\[  K':=\nu^{-1/2}u_{\rho}(m,d-1) \bar{\times} \Pi_p
\]
and $(\pi|_M)/K'$ admits a $M$-filtration whose successive quotients isomorphic to $\pi^{(j)} \bar{\times} \Pi_j$ for $j <p$ (see similar discussions in Section \ref{ss BZ filtration}). Let $G=G_{n}$. Now taking mirabolic product is exact and so one would have, by a long exact sequence argument,
\[ \mathrm{Ext}^i_{G}( (\pi|_M \bar{\times} \pi')|_G, \pi'') \cong \mathrm{Ext}^i_G( (K' \bar{\times} \pi')|_G, \pi'') \]
if we can show that, for all $i$,
\[ \mathrm{Ext}^i_G(((\pi|_M)/K' \bar{\times} \pi'|_G, \pi'') =0
\]

To show the last Ext vanishing, it suffices to show that for each piece of Bernstein-Zelevinsky layer $\tau=\pi^{(j)}\times \Pi_j$ ($j<p$) appearing in $(\pi|_M)/K'$, 
\[\mathrm{Ext}^i_G((\tau \bar{\times} \pi')|_G, \pi'')=0\]
 for any $i$, which indeed follows from:
\begin{align*}
  \mathrm{Ext}^i_G( ((\pi^{(j)}\bar{\times} \Pi_j) \bar{\times} \pi')|_G, \pi'' ) 
\cong & \mathrm{Ext}^i_G( (\nu^{1/2}\pi^{(j)})\times ((\Pi_j \bar{\times} \pi')|_{G_{j+k-1}}), \pi'' ) \\
\cong & \mathrm{Ext}^i_{G_{n_{\pi}-j}\times G_{j+k-1}}((\nu^{1/2}\pi^{(j)})\boxtimes (\Pi_j \bar{\times} \pi'), (\pi'')_{N_{j+k-1}^-}) \\
\cong & 0,
\end{align*}
where the first isomorphism follows from Lemma \ref{lem associative}(1) and (\ref{eqn restriction to g first factor}), the second isomorphism follows from Frobenius reciprocity, and the last isomorphism follows from Lemma \ref{lem speh restrict}(2) with comparing cuspidal supports. The last isomorphism follows from Lemma \ref{lem ext vanishing}.
\end{proof}

\section{Proof of Conjecture \ref{conj ggp orig} (non-Archimedean)} \label{sec proof conj}

The main goal of this section is to prove Conjecture \ref{conj ggp orig} (non-Archimedean)  modulo Proposition \ref{prop bessel transfer} and Proposition \ref{prop preserving quo sp}. Roughly speaking, Lemmas \ref{lem cuspidal supp} and \ref{lem speh rep filtration} reduce to a bottom layer in a filtration and then  Lemma \ref{lem reduct main} reduces the computation of the bottom layer to an inductive case. One also needs a Gan-Gross-Prasad type reduction (Lemma \ref{lem first main reduction detail}) to transfer the study  to the inductive case.

\subsection{Dual restriction}

\begin{proposition} \label{prop bessel transfer}
Let $\pi_1$ and $\pi_2$ be irreducible representations of $G_{n+1}$ and $G_n$ respectively. For $\sigma \in \mathrm{Irr}^c(G_2)$ such that $\sigma$ is not in $\mathrm{cupp}_{\mathbb Z}(\nu^{-1/2}\pi_1^{\vee}) \cup \mathrm{cupp}_{\mathbb Z}(\pi_2)$, and for all $i$, 
\[   \mathrm{Ext}^i_{G_n}(\pi_1|_{G_n}, \pi_2^{\vee}) \cong \mathrm{Ext}^i_{G_{n+1}}(( \pi_2 \times \sigma)|_{G_{n+1}} ,  \pi_1^{\vee}) .
\]
\end{proposition}

The proof of Proposition \ref{prop bessel transfer} will be postponed to Proposition \ref{prop transfer}, where we will prove a more general statement. Note that the additional cuspidal support condition $\sigma \notin \mathrm{cupp}_{\mathbb Z}(\pi_2)$ (c.f. Proposition \ref{prop transfer}) guarantees that $\sigma \times \pi_2 \cong \pi_2 \times \sigma$, while it is not critical in the proof of the GGP conjecture.


\subsection{ Product preserving quotients }

\begin{proposition} \label{prop preserving quo sp}
Let $\rho \in \mathrm{Irr}^{u,c}$. Fix $m,d$. Let $\pi_1$ be a (not necessarily admissible) representation of $G_{n}$. Let $p=n_{\rho}md$. Let $\pi_2$ be an irreducible representation of $G_{n+p}$ such that  any cuspidal representation in $\mathrm{cupp}(\pi_2)$ is either 
\begin{enumerate}
\item lying in $\mathrm{cupp}(u_{\rho}(m,d))=\left\{ \nu^{-(m+d-2)/2}\rho, \ldots, \nu^{(m+d-2)/2}\rho \right\}$; or
\item not lying in $\left\{ \nu^n \nu^{(m+d)/2} \rho \right\}_{n \in \mathbb Z}$. 
\end{enumerate}
 Then if
\[  \mathrm{Hom}_{G_{n+p}}(u_{\rho}(m,d) \times \pi_1, \pi_2) \neq 0 ,
\]
then there exists a non-zero irreducible quotient $\omega$ of $\pi_1$ such that $\pi_2 \cong u_{\rho}(m,d)\times \omega$, moreover, if $\pi_2$ is an irreducible Arthur type representation, then such $\omega$ is also an irreducible Arthur type representation.
\end{proposition}

Proposition \ref{prop preserving quo sp} will be proved as a special case of Corollary \ref{cor stronger}. Proposition \ref{prop preserving quo sp} is only needed for the only if direction.

\subsection{Proof of non-tempered GGP} \label{ss proof GGP}

Recall that $\mathrm{Irr}^{u,c}(G_k)$ is the set of irreducible unitarizable cuspidal representations of $G_k$.

The following two lemmas are the keys for reductions to an inductive case.

\begin{lemma} \label{lem first main reduction detail}
Let $\pi_p$ and $\pi_q$ be Arthur type representations of $G_{n+1}$ and $G_n$ respectively. Write 
\[  \pi_p=\pi_{p,1}\times \ldots \times \pi_{p,r}, \quad \pi_q=\pi_{q,1}\times \ldots \times \pi_{q,s} 
\]
for some Speh representations $\pi_{p,i}, \pi_{q,j}$. Write $\pi_{p,i}=u_{\rho_i}(m_i,d_i)$ and $\pi_{q,j}=u_{\sigma_j}(l_j,e_j)$. Suppose $m_1+d_1\geq m_i+d_i$ and $m_1+d_1 \geq l_j+e_j$ for all $i,j$. Then 
\[ \mathrm{Hom}_{G_n}(\pi_p, \pi_q) \neq 0 \]
 if and only if for any $\widetilde{\sigma} \in \mathrm{Irr}^{u,c}(G_{n_{\rho_1m_1}})$ such that $\widetilde{\sigma} \notin \mathrm{cupp}_{\mathbb{Z}}(\pi_p)\cup \mathrm{cupp}_{\mathbb{Z}}(\nu^{-1/2}\pi_q)$, 
\[   \mathrm{Hom}_{G_n}(u_{\rho_1}(m_1,d_1-1)\times ((\widetilde{\sigma} \times \pi_p')|_{G_a}), \pi_q) \neq 0,
\]
where $\pi_p'=\pi_{p,2}\times \ldots \times \pi_{p,r}$ and $a=n-n_{\rho_1}m_1(d_1-1)$.
\end{lemma}

\begin{proof}

By Lemma \ref{lem short exact seq},
\begin{align} \label{eqn ses mirabolic} 0 \rightarrow \pi_{p,1}|_M \bar{\times} \pi_p'  \rightarrow \pi_p|_M \rightarrow  \pi_{p,1} \bar{\times} (\pi_p'|_M)\rightarrow 0.
\end{align}

Let $n_1=n_{\rho_1}d_1m_1$ and $n'=n-n_1$. Now 
\begin{align*}
\mathrm{Ext}^i_{G_n}((\pi_{p,1}\bar{\times} (\pi_p'|_M))|_{G_n}, \pi_q) & \cong \mathrm{Ext}^i_{G_n}((\nu^{1/2}\pi_{p,1})\times (\pi_p'|_{G_{n'}}), \pi_q) \\
&\cong \mathrm{Ext}^i_{G_{n_1}\times G_{n-n_1}}((\nu^{1/2}\pi_{p,1} )\boxtimes (\pi_p'|_{G_{n'}}), (\pi_q)_{N_{n-n_1}^-}) \\
&=0 ,
\end{align*}
where the first isomorphism follows from (\ref{eqn restriction to g first factor}) and Lemma \ref{lem associative}(1) and the second isomorphism follows from second adjointness of Frobenius reciprocity and the third isomorphism follows by comparing cuspidal support at $\nu^{1/2}\nu^{(d_1+m_1-2)/2}\rho_1$.

Thus long exact sequence argument on (\ref{eqn ses mirabolic}) gives that, for all $i$,
\begin{align} \label{eqn ext transfer 1} \mathrm{Ext}^i_{G_n}((\pi_{p,1}|_M \bar{\times} \pi_p')|_{G_n}, \pi_q ) \cong \mathrm{Ext}^i_{G_n}(\pi_p|_{G_n}, \pi_q ).
\end{align}

Set $u' =\pi_{p,1}^-\cong u_{\rho_1}(m_1,d_1-1)$ and $u''=\nu^{-1/2}u'$. Now Lemma \ref{lem speh rep filtration} gives that
\begin{align} \label{eqn ext transfer 2}  \mathrm{Ext}^i_{G_n}((( u''\bar{\times} \Pi ) \times \pi_p')|_{G_n} , \pi_q ) \cong \mathrm{Ext}^i_{G_n}( (\pi_{p,1}|_M \bar{\times} \pi_p')|_{G_n}, \pi_q) ,
\end{align}
where $\Pi=\Pi_{n_{\rho_1}m_1}$. 

For any $\widetilde{\sigma} \in \mathrm{Irr}^{u,c}(G_{n_{\rho_1}m_1})$ not appearing in $\mathrm{cupp}_{\mathbb Z}(\pi_p) \cup \mathrm{cupp}_{\mathbb Z}(\nu^{-1/2}\pi_q)$,
\begin{align} \label{eqn ext transfer 3}
 \mathrm{Ext}^i_{G_n}(u'  \times ((\widetilde{\sigma} \times \pi_p')|_{G_t}), \pi_q ) \cong \mathrm{Ext}^i_{G_n}(u'  \times ((\Pi \bar{\times} \pi_p')|_{G_t}), \pi_q) \cong  \mathrm{Ext}^i_{G_n}((( u''\bar{\times} \Pi ) \times \pi_p')|_{G_n} , \pi_q ),
\end{align}
where $t=n'+n_{\rho_1}m_1$. Here the first isomorphism follows from Lemma \ref{lem cuspidal supp} and the second isomorphism follows from Lemma \ref{lem associative}(1) and (\ref{eqn restriction to g first factor}),

 By equations (\ref{eqn ext transfer 1}), (\ref{eqn ext transfer 2}) and (\ref{eqn ext transfer 3}) at the case that $i=0$, we obtain the following equivalent statements:
\begin{enumerate}
\item $\mathrm{Hom}_{G_n}( \pi_p|_{G_n}, \pi_q ) \neq 0$;
\item $\mathrm{Hom}_{G_n}(u_{\rho_1}(m_1,d_1-1)  \times ((\widetilde{\sigma} \times \pi_p')|_{G_t}), \pi_q ) \neq 0$ for any $\widetilde{\sigma} \in \mathrm{Irr}^{u,c}(G_{n_{\rho_1}m_1})$ not appearing in $\mathrm{cupp}_{\mathbb Z}(\pi_p) \cup \mathrm{cupp}_{\mathbb Z}(\nu^{-1/2}\pi_q)$.
\end{enumerate}

\end{proof}

\begin{lemma} \label{lem reduct main}
We keep using notations in the previous lemma. We still assume that $m_1+d_1\geq m_i+d_i$ and $m_1+d_1 \geq l_j+e_j$ for all $i,j$. Then 
\[\mathrm{Hom}_{G_n}(\pi_p|_{G_n}, \pi_q)\neq 0\]
if and only if there exists $k$ such that 
\[ \pi_{q,k} \cong u_{\rho_1}(m_1, d_1-1), 
\]
and for any $\widetilde{\sigma} \in \mathrm{Irr}^{u,c}(G_{n_{\rho_1}m_1})$ with $\widetilde{\sigma} \notin \mathrm{cupp}_{\mathbb Z}(\pi_p) \cup \mathrm{cupp}_{\mathbb{Z}}(\nu^{-1/2}\pi_q)$,
\[ \mathrm{Hom}_{G_{n'}}( (\widetilde{\sigma}\times \pi_p')|_{G_{n'}} ,\pi_q') \neq 0,
\]
where $n'=n-n_{\rho_1}m_1(d_1-1)$ and $\pi_q'=\pi_{q,1}\times \ldots \pi_{q,k-1}\times \pi_{q,k+1}\times \ldots \times \pi_{q,s}$.
\end{lemma}

\begin{proof}

We first consider the if direction. Let  $\widetilde{\sigma} \in \mathrm{Irr}^{u,c}(G_{n_{\rho_1}m_1})$ not appear in $\mathrm{cupp}_{\mathbb Z}(\pi_p)\cup \mathrm{cupp}_{\mathbb Z}(\nu^{-1/2}\pi_q)$. By the hypothesis of if direction, $\widetilde{\sigma} \times \pi_p'$ has a quotient $\pi_q'$ (where $\pi_q'$ is defined as in the lemma). Hence, by exactness of parabolic induction, 
\[ u_{\rho_1}(m_1,d_1-1) \times (\widetilde{\sigma} \times \pi_p') \]
 has a quotient 
\[\pi_{q,k}\times \pi_q' \cong u_{\rho_1}(m_1, d_1-1) \times \pi_q' \cong \pi_q . \] Thus, by the if part of Lemma \ref{lem first main reduction detail}, we obtain 
\[ \mathrm{Hom}_{G_n}(\pi_p|_{G_n}, \pi_q) \neq 0 .\]


We now consider the only if direction. Suppose $\mathrm{Hom}_{G_n}(\pi_p|_{G_n}, \pi_q)\neq 0$. By using the only if part of Lemma \ref{lem first main reduction detail}, we have that:
\[  \mathrm{Hom}_{G_n}(u_{\rho_1}(m_1,d_1-1)  \times ((\widetilde{\sigma} \times \pi_p')|_{G_t}), \pi_q ) \neq 0
\]
for some $\widetilde{\sigma} \in \mathrm{Irr}^{u,c}(G_{n_{\rho_1m_1}})$ not in $\mathrm{cupp}_{\mathbb Z}(\pi_p) \cup \mathrm{cupp}_{\mathbb Z}(\nu^{-1/2}\pi_q)$. Here $t=n-n_{\rho_1}m_1(d_1-1)$.

From the condition on $m_1+d_1$, one checks the conditions in Proposition \ref{prop preserving quo sp} and so we can apply it to obtain that
\[ \pi_2 \cong u_{\rho_1}(m_1,d_1-1) \times \omega
\]
for some irreducible Arthur type quotient $\omega$ of $(\widetilde{\sigma} \times \pi_p')|_{G_t}$. Now by uniqueness of factorization of Arthur type representations in terms of Speh representations, there exists some $k^*$ such that 
\[  \pi_{q,k^*} \cong u_{\rho_1}(m_1,d_1-1), \quad \pi_{q,1}\times \ldots \times \pi_{q,k^*-1}\times \pi_{q,k^*+1} \times \ldots \times \pi_{q,s} \cong \omega .
\]
This proves the only if direction.
\end{proof}

\begin{theorem} \label{thm conj ggp}
Conjecture \ref{conj ggp orig} holds for non-Archimedean field $F$. 
\end{theorem}

\begin{proof}
We shall prove the reformulated problem in Section \ref{ss rep the reform}. Let $\pi_p$ and $\pi_q$ be Arthur type representations of $G_{n+1}$ and $G_n$ respectively. We can write as the product of Speh representations i.e.
\[  \pi_p=\pi_{p,1}\times \ldots \times \pi_{p,r} \quad  \mbox{ and } \quad  \pi_q=\pi_{q,1} \times \ldots \times \pi_{q,s} 
\]
such that each $\pi_{p,i}$ (resp. $\pi_{q,j}$) is an (irreducible unitarizable) Speh representation $u_{\rho_i}(m_i, d_i)$ (resp. $u_{\sigma_j}(l_j,e_j)$). Let $N(\pi_p, \pi_q)$ be the total number of factors $\pi_{p,i}$ and $\pi_{p,j}$ which are not cuspidal representation. The basic case is that all $\pi_{p,i}$ and $\pi_{q,j}$ are cuspidal representations i.e. $N(\pi_p,\pi_q)=0$, and so $\pi_p$ and $\pi_q$ are generic. In that case, it is well-known from \cite{JPSS83, GGP12}.

By \cite[Theorem 7.1]{Ta86}, we may and shall assume that for $1 \leq i \leq r$, $1 \leq j \leq s$,
\[ m_1+ d_1 \geq m_i+ d_i  \quad \mbox{ and } \quad l_1+e_1 \geq l_j+e_j . \]
We may also assume that $m_1+d_1 > 2$ or $l_1+e_1 > 2$, and so either $\pi_{p,1}$ or $\pi_{q,1}$ is not cuspidal. Otherwise, it is the basic case.

We now consider two cases:

\noindent
{\bf Case 1:} $m_1+d_1 \geq l_1+e_1$, which implies $\frac{m_1+d_1-2}{2}+\frac{1}{2} >\frac{l_i+e_i-2}{2}$ for all $i$, and so $\nu^{1/2} \nu^{(d_1+m_1-2)/2}\rho_1$ is in $\mathrm{cupp}(\nu^{1/2}\pi_{p,1})$, but is not in the cuspidal support of any $\pi_{q,i}$.  \\

Let 
\[ \pi_p' = \pi_{p,2}\times \ldots \times \pi_{p,r} .
\]


 Let $u=\pi_{p,1}\cong u_{\rho_1}(m_1,d_1)$. We first prove the only if direction and assume that $\mathrm{Hom}_{G_n}(\pi_p, \pi_q) \neq 0$. Using Lemma \ref{lem reduct main}, there exists $\sigma \in \mathrm{Irr}^{u,c}(G_{n_{\rho_1}m_1})$ with $\sigma \notin \mathrm{cupp}_{\mathbb{Z}}(\nu^{-1/2}\pi_q)$ and $k^*$ such that 
\[  \pi_{q,k^*} =u^-, \quad \mbox{ and } \quad \mathrm{Hom}_{G_t}(\sigma \times \pi_p', \pi_q') \neq 0 ,
 \]
where $\pi_q'=\pi_{q,1}\times \ldots \times \pi_{q,k^*-1}\times \pi_{q,k^*+1}\times \ldots \times \pi_{q,s}$ and $t=n-n_{\rho_1}m_1(d_1-1)$. Since $\sigma \times \pi_p'$ is also an Arthur type representation with 
\[ N(\sigma\times \pi_p', \pi_q')=N(\pi_p', \pi_q') <N(\pi_p,\pi_q), \]
 we can apply inductive hypothesis to obtain that
\[ \sigma\times  \pi_p' \cong  \tau_{p,1} \times \ldots \times \tau_{p,k} \times \tau_{q,1}^-\times \ldots \times \tau_{q,l}^-
\]
\[  \pi_q' \cong \tau_{p,1}^-\times \ldots \times \tau_{p,k}^-\times \tau_{q,1} \times \ldots \times \tau_{q,l}
\]
for some Speh representations $\tau_{p,1}, \ldots, \tau_{p,k}, \tau_{q,1}, \ldots, \tau_{q,l}$. Since the product is uniquely determined by the factors of those Speh representations \cite{Ta86} and $\sigma \notin \mathrm{cupp}_{\mathbb Z}(\nu^{-1/2}\pi_q')$, we must have $\tau_{p,i^*} \cong \sigma$ for some $i^*$. Since the products between Speh representations commute, we may simply set $i^*=1$. With $\tau_{p,1}^-=1$, now we have
\[\pi_p \cong u \times \pi_p' \cong u \times \tau_{p,2}\times \ldots \times \tau_{p,k} \times \tau_{q,1}^- \times \ldots \times \tau_{q,l}^-, 
\]
\[ \pi_q \cong u^- \times \pi_q' \cong u^-\times \tau_{p,2}^-\times \ldots \times \tau_{p,k}^- \times \tau_{q,1} \times \ldots \times \tau_{q,l} \]
 as desired.

Now we prove the if direction and so we consider
\[  \pi_p \cong \tau_{p,1} \times \ldots \times \tau_{p,k} \times \tau_{q,1}^-\times \ldots \times \tau_{q,l}^-
\]
and 
\[ \pi_q \cong \tau_{p,1}^-\times \ldots \times \tau_{p,k}^-\times \tau_{q,1}\times \ldots \times \tau_{q,l} 
\]
for some Speh representations $\tau_{p,1}, \ldots , \tau_{p,k}, \tau_{q,1}, \ldots, \tau_{q,l}$. From our choice of $\pi_{p,1}$ and the assumption for Case 1, we must have that, by reindexing if necessary,
\[  \tau_{p,1} \cong \pi_{p,1} .
\]
Then $\tau_{p,1}^- \cong u_{\rho_1}(m_1,d_1-1)$. This implies that 
\[ \tau_{p,2} \times \ldots \times \tau_{p,k}\times \tau_{q,1}^-\times \ldots \times \tau_{q,l}^- \cong \pi_{p,2}\times \ldots \times \pi_{p,r}=\pi_p' ,\]
by unique factorization of Speh representations \cite{Ta86}. Since 
\[ N(\sigma \times \pi_p', \pi_q'')= N(\pi_p', \pi_q'')< N(\pi_p, \pi_q'')\leq N(\pi_p,\pi_q) , \]
 induction gives that for any $\sigma$ of $\mathrm{Irr}^{u,c}(G_{n_{\rho_1}m_1})$, 
\[  \mathrm{Hom}_{G_{n'}}(\sigma \times \pi_p', \pi_q'') \neq 0 ,
\] 
where $\pi_q'' \cong \sigma^-\times \tau_{p,2}^-\times \ldots \times \tau_{p,k}^- \times \tau_{q,1} \times \ldots \times \tau_{q,l}$.
Lemma \ref{lem reduct main} implies that $\mathrm{Hom}_{G_n}(\pi_p|_{G_n}, \pi_q) \neq 0$ as desired. \\



\noindent
{\bf Case 2:} $l_1+e_1 >m_1+d_1$, which implies $\frac{l_1+e_1-2}{2}+\frac{1}{2}> \frac{m_1+d_1-2}{2}$. There are infinitely many unitarizable cuspidal representations of $G_2$, and we can find one satisfying the hypothesis in Proposition \ref{prop bessel transfer}  so that
\begin{align*}
 \mathrm{Hom}_{G_{n+1}}(\pi_q \times \sigma|_{G_{n+1}}, \pi_p) \neq 0 &  \Longleftrightarrow \quad \mathrm{Hom}_{G_n}(\pi_p^{\vee} |_{G_n}, \pi_q^{\vee}) \neq 0 \\
&  \Longleftrightarrow \quad \mathrm{Hom}_{G_{n+1}}(\overline{\pi}_p|_{G_n} ,  \overline{\pi}_q) \neq 0 \\
 & \Longleftrightarrow \quad \mathrm{Hom}_{G_{n+1}}( \pi_p |_{G_{n}}, \pi_q) \neq 0  
.
\end{align*}
Here $\overline{\pi}_p, \overline{\pi}_q$ are complex conjugate representations of $\pi_p, \pi_q$ respectively, and so the last 'if and only if' implication is immediate. The first 'if and only if' implication follows from Proposition \ref{prop bessel transfer} and the second one follows from that $\pi_p, \pi_q$  are unitarizable and so Hermitian self-dual.

We also have that $\pi_q \times \sigma$ is still an Arthur type representation. Note that $N(\pi_q\times \sigma, \pi_p)=N(\pi_p, \pi_q)$. We now use the argument in Case 1 and inductive hypothesis to prove this case, where the role of $\pi_{q,1}$ replaces the one of $\pi_{p,1}$.


\end{proof}



\section{General cases: Bessel, Fourier-Jacobi and Rankin-Selberg models} \label{sec general}

 In this section, we shall generalize the non-tempered GGP to other models of general linear groups. We study some connections between models, which will be continued in Section \ref{sec fj bz theory}. We also improve some previous multiplicity results for Bessel and Fourier-Jacobi models to the Ext versions.

\subsection{Equal rank Fourier-Jacobi models} \label{ss corank 0}
Let $S(F^n)$ be the space of Bruhat-Schwartz functions on $F^n$. For a character $\mu$ of $G_n$, let $\omega_{\mu,0}$ (resp. $\widehat{\omega}_{\mu,0}$) be a $G_n$-representation with underlying space $S(F^n)$ and the $G_n$-action given by
\[   (g.f)(v)=\mu(g)f(g^{-1}v), \quad ( \mathrm{resp.}\quad  (g.f)(v) = \mu(g)f(g^tv) ) .\]

Let $\pi \in \mathrm{Alg}(G_n)$. Since $G_n \setminus M_{n+1} \cong F^n$ as topological spaces, and $\omega_{\mu \nu^{-1/2},0}\otimes \pi$ can be viewed as the space of smooth compactly-supported functions $f: F^n \rightarrow \mu\nu^{-1/2}\pi$ with $G_n$ acting by $(g.f)(v)=g.f(g^{-1}v)$, we have:
\[  \mu \otimes \Lambda(\pi)|_{G_n} \cong \omega_{\mu \nu^{-1/2},0} \otimes \pi
\]
via the natural map for $f \in  \Lambda(\pi)$,
\[  f \mapsto \left(v \mapsto  f(\begin{pmatrix} I_n& v\\ 0 & 1 \end{pmatrix}) \right) .
\]
Set $\zeta^F=\omega_{\nu^{-1/2},0}$ and set $\widehat{\zeta}^F=\widehat{\omega}_{\nu^{1/2},0}$.

\begin{proposition} \label{prop FJ0}
Let $\pi, \pi' \in \mathrm{Alg}(G_n)$. Then there exists a character $\chi$ of $F^{\times}$  such that $\chi \notin  \mathrm{cupp}_{\mathbb Z}(\nu^{-1/2}\pi')$
and, for all $i$,
\[ \mathrm{Ext}^i_{G_n}((\chi \times \pi)|_{G_n}, \pi') \cong \mathrm{Ext}^i_{G_n}( \pi\otimes \zeta^F, \pi') .  
\]
The assertion also holds if we replace for $\zeta^F$ by $\widehat{\zeta}^F$.
\end{proposition}

\begin{proof}
By Lemma \ref{lem short exact seq},
\begin{align} \label{eqn ses fj0}
 0 \rightarrow  \chi|_{M_1} \bar{\times} \pi  \rightarrow (\chi \times \pi)|_M \rightarrow \chi \bar{\times} (\pi|_M) \rightarrow 0 .
\end{align}
Then $\chi|_{M_1} \bar{\times} \pi \cong \Lambda(\pi)$ by the definition of mirabolic induction. By using the above identification, we have 
\begin{align} \label{eqn iso fj0}
\chi|_{M_1}\bar{\times} \pi \cong  \pi \otimes \zeta^F   .
\end{align} 
On the other hand, via Frobenius reciprocity, due to the condition that $\chi \notin \mathrm{cupp}_{\mathbb Z}(\nu^{-1/2}\pi') $, Lemma \ref{lem ext vanishing} implies that for all $i$,
\begin{align} \label{eqn ext fj0}
\mathrm{Ext}^i_{G_n}((\chi \bar{\times} (\pi|_M))|_{G_n} ,\pi')\cong  \mathrm{Ext}^i_{G_n}((\nu^{1/2}\chi) \times (\pi|_{G_{n-1}}) , \pi')=0.
\end{align}
Now standard long exact sequence argument on (\ref{eqn ses fj0}) with (\ref{eqn iso fj0}) and (\ref{eqn ext fj0}) gives, for all $i$, 
\[ \mathrm{Ext}^i_{G_n}(\chi \times \pi, \pi') \cong \mathrm{Ext}^i_{G_n}((\chi|_{M_1}\bar{\times} \pi)|_{G_n} , \pi') \cong \mathrm{Ext}^i_{G_n}(\pi\otimes \zeta^F, \pi').  
\]
The proof for $\widehat{\zeta}^F$ is similar.
\end{proof}

\begin{remark}
From Proposition \ref{prop FJ0}, one can deduce explicit restriction for equal rank Fourier-Jacobi model from the basic restriction from $G_{n+1}$ to $G_n$.  One may also compare with the method using theta correspondence to deduce Fourier-Jacobi models from Bessel models in \cite{GI16} and \cite{At18}.

\end{remark}

\subsection{Bessel, Rankin-Selberg and mixed models} \label{ss bessel}

Let $m_1,m_2, r\geq 0$. Recall that $\bar{\psi}$ is a choice of a non-degenerate character on $F$. Let 
\[  H=\left\{ \begin{pmatrix}  u_1 & x & y \\  & h & z \\  &  & u_2 \end{pmatrix}: \begin{array}{c} u_1 \in U_{m_1}, u_2 \in U_{m_2},  h \in \widetilde{G}_{r+1}, x \in Mat_{m_1\times (r+1)}, \\
 z \in Mat_{(r+1) \times m_2}, y \in Mat_{m_1\times m_2} \end{array} \right\}  \subset G_{m_1+m_2+r+1},
\]
and
\[  \widetilde{G}_{r+1} =\left\{ \mathrm{diag}(1, g)  : g \in G_{r} \right\} .
\]
We shall also write $H^B$ or $H^B_{m_1,m_2,r}$ for $H$.

Let $\varphi_n: U_n\rightarrow \mathbb{C}$ be a non-degenerate character on $U_n$. For example, one may take 
\[  \varphi_n(u)= \bar{\psi}(u_{1,2}+\ldots +u_{n-1,n}) .
\]
 Let $\zeta: H \rightarrow \mathbb{C}$ such that 
\[  \zeta(\begin{pmatrix} u_1 & x & y \\ & g & z \\ & & u_2 \end{pmatrix}) =\varphi_{m_1}(u_1)\varphi_{m_2}(u_2) \bar{\psi}(x_{m_1,1})\bar{\psi}(z_{1,1}) \nu(g)^{-(m_2-m_1)/2},
\]
where $x_{m_1,1}$ (resp. $z_{1,1}$) is the $(m_1,1)$- (resp. $(1,1)$-) coordinate of $x$ (resp. $z$). We shall also sometimes write $\zeta^B$ for $\zeta$. Note that $\nu^{m_2-m_1}$ is the modulus function of $H$ (i.e. a normalizing factor).

Let $U'$ be the unipotent radical of $H$. The orbit by the conjugation action of $(T_{m_1+1} \times G_r \times T_{m_2})U'$ on $\phi$ is the unique dense orbit on the character space of $U'$, where $T_{m_1+1}$ (resp. $T_{m_2}$) be the subgroup of diagonal matrices of $G_{m_1+1}$ (resp. $G_{m_2}$), and as subgroup of $H$ via embedding to the upper (resp. lower) corner.

\begin{remark}
The Bessel subgroup defined in \cite[Sections 12 and 13]{GGP12} is conjugate to $H^B_{m,m,r}$, where $r=n-2m$, for some $m$.


When $m_1=0$ or $m_2=0$, the model is sometimes called a Rankin-Selberg model \cite{CS13, GS20}. We shall also write $H^{R}_{m,r}=H^B_{0,m,r}$ and $\zeta^{R}=\zeta^B$. (The matrix $H^R_{m,r}$ is conjugate to the one in Section \ref{ss intro generalize}.) When $r=0$, the model is Whittaker \cite{Sh74}, and when $m_1=m_2=0$, it is related to the restriction from $G_{n+1}$ to $G_n$ in \cite{AGRS10}.


\end{remark}

There is another formulation of Bessel models, using Bernstein-Zelevinsky functors.  

\begin{proposition} \label{prop bessel bz realize}
 Let $\pi$ be a $G_r$-representation, which extends to a $H$-representation trivially. Let $n=m_1+m_2+r+1$. Then there exist natural isomorphisms:
\begin{align*}
  {}^u\mathrm{ind}_{H^B_{m_1,m_2,r}}^{G_n} \pi \otimes \zeta^{B} \otimes \nu^{(m_2-m_1)} \cong & (\Phi^+)^{m_2+1}(\Pi_{m_1+1} \bar{\times} \pi  )|_{G_n} \\
	\cong & (\Phi^+)^{(m_1+m_2+1)}(\Lambda(\pi))|_{G_n} \\
	\cong & {}^u\mathrm{ind}_{H^R_{m_1+m_2,r}}^{G_n} \pi \otimes \zeta^R \otimes \nu^{m_1+m_2} 
\end{align*}

\end{proposition}

\begin{proof}
The second isomorphism follows from Proposition \ref{prop nontrivial bz}. Note that the last isomorphism is a special case of the first isomorphism. It remains to prove the first isomorphism. Let 
\[ w= \mathrm{diag}(\begin{pmatrix} 0 & I_{r} \\ I_{m_1+1} & 0 \end{pmatrix}, I_{m_2+1}) .
\]
Using induction in stages, the subgroup from which 
\[(\Phi^+)^{m_2+1}(\Pi_{m_1+1} \bar{\times} \pi ))|_{G_n}  \]
 is induced, takes the form:
\[ Q' = \begin{pmatrix}  g &  & & * \\ * & m & * & * \\ & & 1 & * \\ & & &u              \end{pmatrix},
\]
where $g \in G_{r}$, $m \in G_{m_1}$ and $u \in U_{m_2}$, and so $w^{-1}Q'w=H_{m_1,m_2,r}^B$.

The conjugation by the element $w$ then defines a map $\Gamma$ from ${}^u\mathrm{ind}_H^{G_n} \pi \otimes \zeta^B \otimes \nu^{m_2-m_1}$ to $(\Phi^+)^{m_2+1}(\Pi_{m_1+1} \bar{\times} \pi ))|_{G_n}$, as vector spaces, given by
\[ f \mapsto \left( g \mapsto f(w\begin{pmatrix} g & \\ & 1 \end{pmatrix})   \right)
\]
Restricted to the unipotent subgroup $U'$ of $H^B$, $\Gamma(f)$ is copies of character $\zeta^{B}$, while a function in $(\Phi^+)^{m_2+1}(\Pi_{m_1+1}\bar{\times} \pi)$ restricted to $U'$ is copies of another character in the same $B'$-orbit as $\zeta^{B}$, where $B'$ contains matrices of the form $\mathrm{diag}( I_{r} , T_lU_{l})$, where $l=m_1+m_2+1$. Hence there exists $b \in B'$ such that the map $f \mapsto \left(g \mapsto f(bw\begin{pmatrix} g & \\  & 1 \end{pmatrix}) \right)$ is a $G_n$-isomorphism. 

We also remark that the character $\nu^{1/2}$ arisen when restricted to $G_n$ cancels with the character $\nu^{-1/2}$ arisen from the mirabolic induction in $\Pi_{m_1+1} \bar{\times} \pi $. 

\end{proof}

The following result is proved by a similar method as in \cite{GGP12}, also see \cite{CS13}. 

\begin{proposition} \label{prop transfer}
Let $\pi_1, \pi_2$ be representations of $G_{n}$ and $G_r$ respectively. Let $m_1+m_2+r+1=n$. For any irreducible cuspidal representation $\sigma$ of $G_{m_1+m_2+2}$ such that $\sigma \notin \mathrm{cupp}_{\mathbb Z}(\nu^{-1/2}\pi_1^{\vee})$, and for all $i$,
\[ \mathrm{Ext}^i_{H^B_{m_1,m_2,r}}( \pi_1 \otimes \zeta^B, \pi_2^{\vee} ) \cong \mathrm{Ext}^i_{G_n}(\sigma \times \pi_2, \pi_1^{\vee} ) .
\]
\end{proposition}

\begin{remark}
Proposition \ref{prop bessel transfer} is a particular case of Proposition \ref{prop transfer} for $m_1=0$, $m_2=0$ and $r=n-1$.
\end{remark}

\begin{proof}
By Lemma \ref{lem short exact seq} again,
\[   0 \rightarrow \sigma|_M \bar{\times} \pi_2   \rightarrow (\sigma \times \pi_2)|_M \rightarrow \sigma \bar{\times} (\pi_2|_M) \rightarrow 0 
\] 
Since $\sigma$ is cuspidal, $\sigma|_M \cong \Pi_{m_1+m_2+2}$. Now with Propositions \ref{prop nontrivial bz} and \ref{prop bessel bz realize}, 
\[  {}^u\mathrm{ind}_H^{G} \pi_2\otimes \zeta^B = (\sigma|_M \bar{\times} \pi_2)|_{G_n}
\]
Again the cuspidal condition guarantees that, for all $i$,
\[  \mathrm{Ext}^i_{G_n}((\nu^{1/2}\sigma) \times (\pi_2|_{G_{r-1}}), \pi_1^{\vee}) =0 .
\]
Now similar argument with the proof of Proposition \ref{prop FJ0}, one reduces to, for all $i$,
\begin{align*}
\mathrm{Ext}^i_{G_n}(\sigma \times \pi_2, \pi_1^{\vee} ) \cong & \mathrm{Ext}^i_{G_n}((\sigma|_M \bar{\times} \pi_2)|_{G_n}, \pi_1^{\vee})  \\
 \cong &  \mathrm{Ext}^i_{G_n}( {}^u\mathrm{ind}_H^{G_n} \pi_2\otimes \zeta^B \otimes \nu^{m_2-m_1} , \pi_1^{\vee} )  \\
 \cong & \mathrm{Ext}^i_{G_n}(\pi_1 , {}^u\mathrm{Ind}_H^{G_n} (\pi_2 \otimes \zeta^B)^{\vee}) \quad \mbox{ (taking duals) } \\
\cong & \mathrm{Ext}^i_H(\pi_1, (\pi_2 \otimes \zeta^B)^{\vee} ) \quad \mbox{ (Frobenius reciprocity)} \\
\cong & \mathrm{Ext}^i_H(\pi_1 \otimes \zeta^B, \pi_2^{\vee}) \quad \mbox{ (taking duals) }
\end{align*}
 For the last three isomorphism, also see \cite{Pr18}.



\end{proof}





\subsection{Fourier-Jacobi models} \label{ss fourier}

 Let $S(F^r)$ be the space of Bruhat-Schwartz functions on $F^r$. Let $W=F^r$ and let $K_r$ be the Heisenberg group i.e. $K_r$ is the group isomorphic to $F\oplus W \oplus W^{\vee}$ with the multiplication:
\[ (a, v, w) \cdot (a',v',w')= (a+a'+w^tv' , v+v', w+w') .
\]

Define
\[ H_r' = \left\{ \begin{pmatrix} 1 & w^t & a\\  & g & v \\  & & 1 \end{pmatrix} :  v,w \in F^r, a \in F, g \in G_r \right\} \]
and so $H_r' \cong G_r \ltimes K_r$. Here we identify $W$ and $W^{\vee}$ with $F^r$ so that $y(x)=y^tx$ for $x \in W$ and $y\in W^{\vee}$.

 Fix a character $\mu$ of $G_r$. Let $\lambda$ be a non-trivial character on $F$. The Weil representation $\omega_{\mu, \lambda}$ of $K_r$ associated to $\lambda$ is the representation with underlying space as $S(W)$ with the action of $K_r$ given by: for $f \in S(W) \cong S(F^r)$,
\[  ((a, v, w). f)(x) = \lambda(a-w^tx-w^tv)f(x+v) .
\]
and for $f \in S(W^{\vee}) \cong S(F^r)$, 
\[  ((a,v,w).f)(y)= \lambda(a+y^tv)f(y+w) .
\]

This extends $\omega_{\mu, \lambda}$ to a $H_r'$-representation $\widetilde{\omega}_{\mu, \lambda}$ (resp. $\widehat{\omega}_{\mu, \lambda}$) given by: for $g \in G_r$, and $f \in S(W)$ (resp. $f \in S(W^{\vee})$),
\[   (g.f)(x)=\mu(g) \cdot f(g^{-1}.x) , \quad \mbox{ (resp.  $(g.f)(y) = \mu(g) \cdot f(g^t.y)$) .}
\]

\begin{lemma} \label{lem identify 1} 
Let $\pi \in \mathrm{Alg}(G_r)$, extend trivaily to $H_r'$. Then 
\[ \pi \otimes \widehat{\omega}_{\mu, \bar{\psi}} \cong {}^u\mathrm{ind}_{H_{0,1,r}^B}^{H_r'}\mu \pi \otimes (\zeta^B \otimes \nu^{1/2})  .
\]

\end{lemma}

\begin{proof}
We can identify $ \nu^{-1/2} \pi \otimes \widehat{\omega}_{\mu, \bar{\psi}}$ with the space of smooth compactly supported functions $f: F^r \rightarrow \nu^{-1/2}\pi$ with the action given by $(g.f)(y)=g.f(g^ty)$. Since $H_{0,1,r}^B\setminus H_r' \cong F^r$ as topological spaces, the identification gives a map $\mathcal F:  \pi \otimes \widehat{\omega}_{\mu, \bar{\psi}} \rightarrow {}^u\mathrm{ind}_{H_{0,1,r}^B}^{H_r'}\mu \pi \otimes (\zeta^B \nu^{1/2})$ given by
\[  \mathcal F(f)(y) = f(\begin{pmatrix} 1 & y^t & \\ & I_r & \\ & & 1 \end{pmatrix}) 
\]
\end{proof}

Now we consider general Fourier-Jacobi models. Let $m_1, m_2 \geq 1$. Let $H$ (resp. $U_H$) be the subgroups of $G_{m_1+m_2+r}$ containing all elements of the form:
\[ \begin{pmatrix}  u_1 & x & y \\   & h & z \\ & & u_2 \end{pmatrix} \quad (\mbox{resp.} \begin{pmatrix} u_1 & x & y \\ & I_{r+2} & z \\ & & u_2 \end{pmatrix}) 
\]
with $u_1 \in U_{m_1-1}$, $u_2 \in U_{m_2-1}$, $h \in H'_r$, $x \in Mat_{m_1-1, r+2}$, $z \in Mat_{r+2, m_2-1}$ and $y \in Mat_{m_1-1, m_2-1}$. We shall also write $H^F_{m_1,m_2,r}$ or $H^F$ for $H$. Note that we have $H \cong H_r' \ltimes U_H$. In the case that $m_1=m_2=1$, it recovers the notion for $H_r'$. 

We now extend the representations $\omega_{\mu, \lambda}$ of $H_r'$ to be a representation of $H$, still denoted $\omega_{\mu, \lambda}$ by abuse of notation, whose underlying space is $S(F^r)$ with the action, for $f \in S(F^r)$,
\begin{align} \label{eqn action original}
  \begin{pmatrix}  u_1 & x & y \\   & h & z \\ & & u_2 \end{pmatrix}.f= \varphi_{m_1}(u_1) \varphi_{m_2}(u_2) (h.f)   .
\end{align}
We similarly define the representation $\widehat{\omega}_{\mu, \lambda}$

Set 
\[\zeta=\zeta^F_{m_1,m_2,r,\lambda}=\zeta^F =\nu^{(m_1-m_2)/2} \widetilde{\omega}_{\nu^{-1/2}, \lambda} ,\]
 and 
\[\widehat{\zeta}=\widehat{\zeta}^F_{m_1,m_2,r,\lambda}=\widehat{\zeta}^F=\nu^{(m_1-m_2)/2}\widehat{\omega}_{\nu^{1/2},\lambda} .\]
 Again when $m_1=m_2$, it is the original notion of Fourier-Jacobi model in \cite[Section 15]{GGP12}. The restriction problems involving $\zeta^F$ (and $\widehat{\zeta}^F$) (i.e. $\mathrm{Hom}_H(\pi_1\otimes \zeta^F, \pi_2)$) do not depend on a choice of $\lambda$.

\begin{proposition} \label{prop transfer 2}
Let $n=m_1+m_2+r$ with $m_1, m_2,r \geq 1$. Let $\pi \in \mathrm{Alg}(G_r)$.
\[  {}^u\mathrm{ind}_{H^B_{m_1-1,m_2,r}}^{G_n} \pi \otimes \zeta^B \otimes \nu^{m_2-m_1+1}  \cong {}^u\mathrm{ind}_{H^F_{m_1,m_2,r}}^{G_n} \pi \otimes \widehat{\zeta}^F \otimes \nu^{m_2-m_1}. 
\]
\end{proposition}

\begin{proof}
From constructions, $\zeta^B|_{U_H} \cong \widehat{\zeta}^F|_{U_H}$. Note that $H_r'$ normalizes $U_H$ and the conjugation action of $H_r'$ on $\widehat{\zeta}^F|_{U_H}$ is trivial.  One can extend the identification in Lemma \ref{lem identify 1} to, as $H^F$-representations, 
\[ \pi \otimes (\widehat{\zeta}^F \otimes \nu^{-1/2}\nu^{(m_2-m_1)/2}) \cong {}^u\mathrm{ind}_{H^B}^{H^F} \pi \otimes (\zeta^B \otimes \nu^{(m_2-m_1+1)/2}).
\]
Now applying induction from $H^F$ to $G$, an induction by stages gives the lemma.
\end{proof}
 
In view of Propositions \ref{prop nontrivial bz},  \ref{prop bessel bz realize} and \ref{prop transfer 2}, we can prove in a similar way as in the proof of Proposition \ref{prop transfer} (also similar to the proof of Proposition \ref{prop FJ0}). We omit the details.

\begin{proposition} \label{prop transfer fourier bessel}
Let $m_1, m_2,r \geq 1$. Let $n=m_1+m_2+r$. Let $\pi_1 \in \mathrm{Alg}(G_n)$, and let $\pi_2 \in \mathrm{Alg}(G_r)$. Then, for any cuspidal representation $\sigma$ of $G_{n+1-r}$ with $\sigma \notin \mathrm{cupp}_{\mathbb{Z}}(\pi_2) \cup \mathrm{cupp}_{\mathbb Z}(\nu^{-1/2}\pi_1^{\vee})$, and for any $i$,
\[  \mathrm{Ext}^i_{H_{m_1,m_2,r}^F}(\pi_1 \otimes \widehat{\zeta}^F, \pi_2^{\vee}) \cong \mathrm{Ext}^i_{G_n}(\sigma \times \pi_2, \pi_1^{\vee})
\]
\end{proposition}


Now we give a connection of the two notions $\zeta^F$ and $\widehat{\zeta}^F$.

\begin{proposition} \label{prop transfer dual}
Let $m_1,m_2,r \geq 1$ and let $n=m_1+m_2+r$. Let $\pi_1 \in \mathrm{Alg}(G_{n})$ and let $\pi_2 \in \mathrm{Alg}(G_{r})$. For all $i$,
\[  \mathrm{Ext}^i_H(\pi_1\otimes \zeta^F , \pi_2^{\vee} ) \cong \mathrm{Ext}^i_{\widetilde{H}}(\theta(\pi_1) \otimes \widehat{\zeta}^F, \theta(\pi_2)^{\vee} ),
\]
where $\widetilde{H} = H_{m_2,m_1,r}^F =w\theta(H)w^{-1}$. Here $w$ is the matrix with all $1$ in the antidiagonal and $0$ elsewhere.
\end{proposition}

\begin{proof}
Let $\theta^w$ be the action of $\theta$ followed by the conjugation of $w$. We use the same $\theta^w$ for the induced map on representations. Note that $\theta^w(\pi_1) \cong \theta(\pi_1)$ as $G_n$-representations, $\theta^w(\pi_2^{\vee}) \cong \theta(\pi_2^{\vee})\cong \theta(\pi_2)^{\vee}$ as $G_r$-representation, and $\theta^w(\zeta^F_{\lambda}) \cong \widehat{\zeta}^F_{\lambda^{-1}}$.
\end{proof}

\subsection{Restrictions}

We state the multiplicity one and finiteness for the general cases (c.f. \cite{GGP12}):
\begin{corollary} 
Let $(H,\zeta)$ be any pair described in Sections \ref{ss corank 0}, \ref{ss bessel} and \ref{ss fourier}. Let $\pi_1$ and $\pi_2$ be irreducible representations of $G_n$ and $G_r$ respectively. Then
\[ \mathrm{dim}~\mathrm{Hom}_{H}(\pi_1\otimes \zeta, \pi_2)  \leq 1,
\]
and for all $i$,
\[ \mathrm{dim}~\mathrm{Ext}^i_{H}(\pi_1\otimes \zeta, \pi_2) <\infty
\]
\end{corollary}

\begin{proof}
Proposition \ref{prop transfer} reduces to the case that restricting from $G_{n+1}$ to $G_n$, which is proved in \cite{AGRS10} for Hom and follows from \cite{Pr18, AS18} for higher Ext. 
\end{proof}

\begin{theorem} \label{thm general cases}
Let $(H,\zeta)$ be any pair described in Sections \ref{ss corank 0}, \ref{ss bessel} and \ref{ss fourier}. Let $\pi_M$ and $\pi_N$ be Arthur type representations of $G_n$ and $G_r$ respectively. Then 
\[  \mathrm{Hom}_{H }(\pi_M \otimes \zeta, \pi_N) \neq 0 
\]
if and only if their associated Arthur parameters $M_A$ and $N_A$ are relevant.
\end{theorem}

\begin{proof}
When $r=0$, the model is Whittaker and it is well-known. Assume $r \geq 1$. For the Bessel models, this follows from Proposition \ref{prop transfer} (in which we choose $\sigma$ to be a unitarizable cuspidal representation) and Theorem \ref{thm conj ggp}. For the Fourier-Jacobi models, using Propositions \ref{prop transfer fourier bessel} and \ref{prop transfer dual}, it is equivalent to show that $\theta(\pi_M)$ and $\theta(\pi_N)$ have relevant Arthur parameters. By the Gelfand-Kazhdan isomorphism \cite{BZ76}, $\theta(\pi_M)\cong \pi_M^{\vee}$ and $\theta(\pi_N)\cong \pi_N^{\vee}$. Thus now the statement follows from that $\pi_M,\pi_N$ have relevant Arthur parameter if and only if $\pi_M^{\vee}, \pi_N^{\vee}$ have relevant Arthur parameter.
\end{proof}

\subsection{A filtration on parabolically induced modules}

The notion of those models also provide a convenient way to state the following filtration, which can be regarded as a systematic tool for studying restriction of parabolically induced representations (e.g. \cite{Ch21}). For example, one may use it to replace some arguments in Lemmas \ref{lem cuspidal supp} and \ref{lem speh rep filtration}.

\begin{proposition} \label{cor filtration par ind mod}
Let $\pi_1 \in \mathrm{Alg}(G_{n_1})$ and let $\pi_2 \in \mathrm{Alg}(G_{n_2})$. Let $n_1+n_2=n+1$. Then there exists a filtration on $(\pi_1 \times \pi_2)|_{G_n}$:
\[  0 \subset \tau_n \subset \tau_{n-1} \subset \ldots \subset \tau_1 \subset \tau_0 = \pi_1 \times \pi_2
\]
such that 
\[ \tau_0/\tau_{1} \cong (\nu^{1/2}\pi_1) \times (\pi_2|_{G_{n_2-1}})
\]
and
\[  \tau_{1}/\tau_2 \cong  \pi_1^{[1]} \times (\pi_2 \otimes \zeta^F) ,
\]
and for $k \geq 2$,
\[  \tau_{k}/\tau_{k+1} \cong  \pi_1^{[k]} \times {}^u\mathrm{ind}_{H^R_{k-2, n_2}}^{G_{n_2+k-1}} \pi_2 \otimes \zeta^R \otimes \nu^{k-2}.
\]
\end{proposition}

\begin{proof}
This is a consequence of Lemma \ref{lem short exact seq}, Bernstein-Zelevinsky filtrations (for some details, see Lemma \ref{lem speh rep filtration}), and Proposition  \ref{prop nontrivial bz}.

\end{proof}

\subsection{Consequnce on Ext-branching law}

We also deduce the Ext-analog result in \cite{CS18} for Bessel and Fourier-Jacobi models.

\begin{corollary}
Let $(H,\zeta)$ be any pair described in Sections \ref{ss corank 0}, \ref{ss bessel} and \ref{ss fourier}. Let $\pi_1$ and $\pi_2$ be irreducible generic representations of $G_{n}$ and $G_r$ respectively. Then, for all $i \geq 1$,
\[  \mathrm{Ext}^i_{H}(\pi_1 \otimes \zeta, \pi_2) =0 .
\]
\end{corollary}

\begin{proof}
The case of the Bessel model for $r=n-1$ is proved in \cite{CS18}. The general case now follows from the case in \cite{CS18} and Propositions \ref{prop FJ0}, \ref{prop transfer}, \ref{prop transfer fourier bessel} and \ref{prop transfer dual}. (We remark that for a suitable choice of $\sigma \in \mathrm{Irr}^{u,c}$, $\sigma \times \pi_1$ is still generic.)
\end{proof}

\section{Fourier-Jacobi models and Bernstein-Zelevinsky theory} \label{sec fj bz theory}

In Section \ref{sec general}, we apply Bernstein-Zelevinsky theory to obtain isomorphisms of models. In this section, we further investigate the isomorphisms, and a goal is to obtain Corollary \ref{cor all same models}.
\subsection{Fourier-Jacobi model and its dual}

Recall that $\zeta_F$ and $\widehat{\zeta}_F$ are defined in Sections \ref{ss corank 0} and \ref{ss fourier}. We first consider the equal rank case.

\begin{proposition} \label{cor iso fj dual}
In the equal rank case, $ \zeta^F \cong \widehat{\zeta}^F$ as $G_n$-representations.
\end{proposition}

\begin{proof}

Let $a \in F^{\times}$. For $f \in S(F^r)$, define the Fourier transform:
\begin{align} \label{eqn fourier transform}
 \widehat{f}(y) = \int_{F^r} \bar{\psi}(ay^tx) f(x) dx ,
\end{align}
which is still smooth and compactly supported, and so in $S(F^r)$, and we regard it as a map from $\zeta_F$ to $\widehat{\zeta}_F$. It is straightforward to check well-definedness of the map. One can define the inverse similarly. 

\end{proof}

The above proposition can also be proved by considering the Hecke algebra realization at each Bernstein component, and deduced from left and right filtrations in \cite{CS18, Ch19}.

\begin{proposition} \label{prop fourier transform}
We use the Fourier-Jacobi models in Section \ref{ss fourier} and the Fourier transform defined in (\ref{eqn fourier transform}). The map $\Omega: S(F^r)\rightarrow S(F^r)$ by  $f \mapsto (y \mapsto \widehat{f}(-a^{-1}y))$ defines a $H_r'$-map from $\zeta_{1,1,r, \bar{\psi}}^F$ to $\widehat{\zeta}_{1,1,r,\bar{\psi}}^F$.
\end{proposition}
\begin{proof}
It follows from a straigthforward computation as in previous proposition. We omit the details.

\end{proof}

We summarize the identifications as follow:

\begin{corollary} \label{cor all same models}
Let $\pi \in \mathrm{Alg}(G_r)$. For $m_1, m_2, r \geq 1$,
\begin{align*}
  {}^u\mathrm{ind}_{H_{m_1,m_2,r}^F}^{G_n} \pi \otimes \zeta^F \otimes \nu^{m_2-m_1}  \cong & {}^u\mathrm{ind}_{H_{m_1,m_2,r}^F}^{G_n} \pi \otimes \widehat{\zeta}^F \otimes \nu^{m_2-m_1} \\
	\cong & {}^u\mathrm{ind}_{H^B_{m_1-1,m_2,r}}^{G_n} \pi \otimes \zeta^B \otimes \nu^{m_2-m_1+1}  \\
	\cong & {}^u\mathrm{ind}_{H_{m_1-1+m_2, r}^{R}}^{G_n} \pi \otimes \zeta^{R} \otimes \nu^{m_1+m_2-1}
\end{align*}

\end{corollary}

\begin{proof}
 Proposition \ref{prop fourier transform} implies that, as $H_{m_1,m_2,r}^F$-representations, $\pi \otimes \zeta^F \cong \pi \otimes \widehat{\zeta}^F$ and hence we obtain the isomorphism. Now the remaining isomorphisms follow from Proposition \ref{prop bessel bz realize}.
\end{proof}

\begin{remark} \label{rmk fourier transfor bz theory}
As we have seen, there is a more direct connection via (\ref{prop transfer 2}) and the first isomorphism of Proposition \ref{prop bessel bz realize}:
\begin{align} \label{eqn fj iso1}  {}^u\mathrm{ind}_{H_{m_1,m_2,r}^F}^{G_n} \pi \otimes \widehat{\zeta}^F \otimes \nu^{m_2-m_1} \cong (\Phi^+)^{m_2+1}(\Pi_{m_1} \bar{\times} \pi ))|_{G_n} ,
\end{align}
and similarly, we can obtain:
\begin{align} \label{eqn fj iso2} {}^u\mathrm{ind}_{H_{m_1,m_2,r}^F}^{G_n} \pi \otimes \zeta^F \otimes \nu^{m_2-m_1} \cong (\Phi^+)^{m_2}(\Pi_{m_1+1} \bar{\times} \pi ))|_{G_n} .
\end{align}
The LHS of (\ref{eqn fj iso1}) and (\ref{eqn fj iso2}) are connected via Fourier transform in Proposition \ref{prop fourier transform}, while the RHS of (\ref{eqn fj iso1}) and (\ref{eqn fj iso2}) can be directly connected via Bernstein-Zelevinsky theory (Proposition \ref{prop nontrivial bz}).
\end{remark}

\begin{corollary}
Let $m_1,m_2,r\geq 1$. Let $\pi_1 \in \mathrm{Alg}(G_{m_1+m_2+r})$ and let $\pi_2 \in \mathrm{Alg}(G_r)$. There are natural isomorphisms:
\begin{align*}
 \mathrm{Ext}^i_{H^F_{m_1,m_2,r}}(\pi_1 \otimes \zeta^F, \pi_2^{\vee}) & \cong \mathrm{Ext}^i_{H^F_{m_1,m_2,r}}(\pi_1 \otimes \widehat{\zeta}^F, \pi_2^{\vee}) \\
&\cong \mathrm{Ext}^i_{H^B_{m_1-1,m_2,r}}(\pi_1\otimes \zeta^B, \pi_2^{\vee}) \\
&\cong \mathrm{Ext}^i_{H^R_{m_1+m_2-1,r}}(\pi_1 \otimes \zeta^R, \pi_2^{\vee}) 
\end{align*}
\end{corollary}

\begin{example}
We consider the equal rank Fourier-Jacobi model. For a generalized Steinberg representation $\mathrm{St}(\Delta)$ of $G_n$, we expect that $\mathrm{St}(\Delta) \otimes \zeta^F$ is projective and is isomorphic to the Gelfand-Graev representation of $G_n$ (c.f. \cite{CS18,CS19,Ch19}).
\end{example}

\section{Ext-branching laws} \label{ss ext branching}

\subsection{Conjecture on Ext-branching laws}

We formulate the following question about Ext-branching laws stated in the form of a conjecture, which gives a possible generalization of some observations in \cite{GGP19}. 

\begin{conjecture} \label{conj ext sum}
Let $\pi_M$ and $\pi_N$ be Arthur type representations of $G_{n+1}$ and $G_n$ respectively. Then, for any $i$,
\[  \mathrm{Ext}^i_{G_n}(\pi_M, \pi_N) \cong \bigoplus_{k} \mathrm{Ext}^i_{G_{n+1-k}}(\pi_M^{[k]}, {}^{(k-1)}\pi_N) .
\]
\end{conjecture}

It would be an interesting question to give a more precise formulation on predicting non-vanishing Ext-groups of Arthur type representations (see \cite[Proposition 5.7, Remark 5.8]{GGP19}).

We remark that the appearance of left derivatives in the second spot comes from the second adjointness property of an induction in the Bernstein-Zelevinsky filtration (see e.g. \cite[Lemma 2.4]{CS18}). We shall give few examples of the above conjecture below.

\subsection{$\mathrm{Hom}$-branching}

\begin{example} \label{ex all generic}
Let $\pi_M$ and $\pi_N$ be generic Arthur type representations of $G_{n+1}$ and $G_n$ respectively. Then $\pi_M=\mathrm{St}(\mathfrak m)$ and $\pi_N=\mathrm{St}(\mathfrak n)$ for some multisegments $\mathfrak m$ and $\mathfrak n$. A computation via comparing cuspidal support gives that, for $i \neq 0$ or $k\neq n$, 
\[  \mathrm{Ext}^i_{G_n}(\pi_M^{[k+1]}, {}^{(k)} \pi_N) =0 .
\]
Then 
\[  \mathrm{Hom}_{G_n}(\pi_M, \pi_N) \cong \mathrm{Hom}_{G_0}(\pi_M^{[n+1]}, {}^{(n)}\pi_N) \cong \mathbb{C} .
\]
This recovers the Ext-vanishing theorem \cite{Pr18, CS18} and the multiplicity one theorem \cite{AGRS10, SZ12} in this special case.

We remark that the same formulation of Conjecture \ref{conj ext sum} for arbitrary respective generic representations $\pi_M$ and $\pi_N$ of $G_{n+1}$ and $G_n$ is not true.
\end{example}

\begin{example}
Let $\pi_M$ and $\pi_N$ be Arthur type representations of $G_{n+1}$ and $G_n$ respectively. Suppose their associated Arthur parameters are relevant. Write those Arthur parameters $M_A$ and $N_A$ as (\ref{eqn relevant MA}) and (\ref{eqn relevant NA}) respectively.

Then, Conjecture \ref{conj ext sum} for Hom-case follows from (Theorem \ref{thm conj ggp} and) the following
\begin{align} \label{eqn unique derivative}
  \mathrm{Hom}_{G_{n+1-k}}(\pi_M^{[k+1]}, {}^{(k)}\pi_N) \neq 0 \quad \Longleftrightarrow \quad k=\sum_{d=0}^r \mathrm{dim}~ M_d^+-1=\sum_{d=0}^s\mathrm{dim}~M_d^- .
	\end{align}
The direction $\Leftarrow$ is easy. For the $\Rightarrow$ direction, one may hope to compute the Hom of those derivatives directly while it seems it have not been done so far. We shall sketch how to modify the proof of Theorem \ref{thm conj ggp} to see (\ref{eqn unique derivative}). We use all the notations in the proof of Theorem \ref{thm conj ggp}, and in particular, write
\[ \pi_M=\pi_p = \pi_{p,1} \times \ldots \times \pi_{p,r} , \mbox{ and} \quad\pi_N=\pi_q = \pi_{q,1} \times \ldots \times \pi_{q,s} .
\]
The basic case is again all $\pi_{p,i}, \pi_{q,j}$ are cuspidal, which is included in Example \ref{ex all generic}. Since taking duals behaves well with derivatives, Case 2 (in Theorem \ref{thm conj ggp})  follows from Case 1. 

We only consider Case 1. Again, we use the short exact sequence:
\[  0 \rightarrow \pi_{p,1}|_M \bar{\times} \pi_p'  \rightarrow \pi_p|_M \rightarrow \pi_{p,1} \bar{\times} (\pi_p'|_M) \rightarrow 0.
\]
Note that any Bernstein-Zelevinsky layer of $\pi_{p,1}\times (\pi_p'|_M)$ cannot contribute a non-zero Hom with $\pi_2$, by comparing cuspidal support. With similar consideration as in  Theorem \ref{thm conj ggp}, the only Bernstein-Zelevinsky layer that can contribute non-zero Hom with $\pi_1$ takes the form $(\nu^{-1/2}\pi_{p,1}^-) \bar{\times} (\Pi \times \pi_p')$, which can then be transferred to study the layers in $(\nu^{-1/2}\pi_{p,1}^-) \bar{\times} ((\sigma \times \pi_p')|_M)$. Now one applies induction on the unique layer in $(\sigma \times \pi_p')|_M$ that can contribute non-zero Hom with $\pi_q$, which gives the required integer in (\ref{eqn unique derivative}).


\end{example}







\subsection{Generic representations}

An irreducible representation $\pi$ of $G_n$ is generic if it admits a Whittaker model or equivalently $\pi^{(n)}\neq 0$. The classification of generic representations of $G_n$ in terms of segments is obtained in \cite[Section 9]{Ze80}. We now treat the case that when one of Arthur type representations is tempered and hence is generic. Compared to the Hom-case (also see \cite[Theorem 5.1]{Gu18} and \cite[Corollary 2.8]{Ch19}), a wider class of Arthur type representations can be paired to obtain non-vanishing higher Ext-groups.

\begin{theorem} \label{thm ext generic}
Let $\pi_p$ and  $\pi_q$ be Arthur type representations of $G_{n+1}$ and $G_n$ respectively. Suppose at least one of $\pi_p$ or $\pi_q$ is generic. 
\begin{enumerate}
\item Then there exists at most one integer $j^*$ such that 
\[  \mathrm{Ext}^{i}_{G_n}(\pi_p^{[j^*]}, {}^{(j^*-1)}\pi_q) \neq 0 
\]
for some $i$ and furthermore if $\pi_p$ (resp. $\pi_q$) is not generic, then $j^*$ (resp. $j^*-1$) is the level of $\pi_p$ (resp. $\pi_q$); and if both $\pi_p$ and $\pi_q$ are generic, then $j^*=n+1$. 
\item Suppose $\pi_p$ is generic. Then such $j^*$ in (1) exists if and only if $ \pi_p \cong \pi_q^{gen} \times \pi'$, where $\pi_q^{gen}$ is the generic representation with same cuspidal support as $\pi_q^-$, and $\pi'$ is some irreducible generic (tempered) representation.
\item Suppose $\pi_q$ is generic. An analogous statement holds by switching the role of $\pi_p$ and $\pi_q$ in (2). 
\end{enumerate}
\end{theorem}

\begin{proof}
We first consider (1). Assume that $\pi_p$ is not a generic representation and $\pi_q$ is a generic representation. Let
\[  \pi_p = \pi_{p,1}\times \ldots \times \pi_{p,r},\quad  
 \pi_q = \pi_{q,1} \times \ldots \times \pi_{q,s} .
\]
where each $\pi_{p,i}$ is a Speh representation and each $\pi_{q,j}$ is isomorphic to $\mathrm{St}(\Delta_{q,j})$ for some segment $\Delta_{q,j}$. 

Then the $i$-th derivative $\pi_p^{[i]}$ takes the form, for $i_1+\ldots +i_r=i$, 
\[ \nu^{1/2}( \pi_{p,1}^{(i_1)} \times \ldots \times \pi_{p,r}^{(i_r)})
\]
For each representation $\omega$, we call the cuspidal support $\mathrm{cupp}(\omega)$ is 
\begin{enumerate}
\item $G$-positive (resp. $G$-negative) if for each irreducible unitarizable cuspidal representation $\sigma$ and for all positive (resp. negative) integer $a$, the multiplicity of $\nu^a\sigma$ in $\mathrm{cupp}(\omega)$ is at least that of $\nu^{-a} \sigma$.
\item balanced if $\mathrm{cupp}(\omega)$ is both $G$-positive and $G$-negative. 
\end{enumerate}

Write $\pi_{p,j}=u_{\rho}(m,d)$. Note that for any $i$ such that $\pi_{p,j}^{(i)}$ is non-zero, $\mathrm{cupp}(\nu^{1/2}\pi_{p,j}^{(i)})=\mathrm{cupp}(\pi_{p,j}^-)+\mathrm{cupp}(\mathrm{St}(\Delta))$ for $\Delta=[\nu^{(m-d)/2+k+1/2}\rho, \nu^{(m+d-2)/2+1/2}\rho]$, where $k=i/n_{\rho}$. Since $\mathrm{cupp}(\pi_{p,j}^-)$ is balanced and $\mathrm{cupp}(\mathrm{St}(\Delta))$ is $G$-positive, $\nu^{1/2}\pi_{p,j}^{(i)}$ is $G$-positive for any $i$ and is balanced only if $i$ is the level of $\pi_{p,j}$.

On the other hand, since $\pi_{q,j}$ is a generalized Steinberg representation, ${}^{(i-1)}\pi_{q,j}$ is $G$-negative for all $i$ and is balanced only if $i=0$ or $i$ is the level of $\pi_{q,j}$. Thus $\mathrm{cupp}(\pi_1^{[i]}) =\mathrm{cupp}({}^{(i-1)}\pi_2)$ only if $i$ is the level of $\pi_1$ as desired.

Other cases are similar, or one may use Lemma \ref{prop bessel transfer}.

We now consider (2). The above discussion proves the only if direction by the cuspidal support consideration. It remains to prove the if direction.  From above discussion, it suffices to show that $\mathrm{Ext}^i_{G_{n+1-j^*}}(\pi_q^{gen}, \pi_q^-) \neq 0$ for some $i$. Since $\pi_q^{gen}$ is generic, we can write $\pi_q^{gen}$ as 
\[  \pi_q^{gen} \cong \mathrm{St}(\Delta_1)\times \ldots \times \mathrm{St}(\Delta_k) ,
\]
where each $\Delta_i=[\nu^{-a}\rho, \nu^a \rho]$ for some $a$ and some unitarizable representation $\rho$. For simiplicity, set $\pi'=\pi_q^-$. Thus via Frobenius reciprocity, it suffices to show 
\[ (*)\quad E_i:= \mathrm{Ext}^i(\mathrm{St}(\Delta_1)\boxtimes \ldots \boxtimes \mathrm{St}(\Delta_k), \pi'_{N^-}) \neq 0,
\]
where $N^-$ is the opposite unipotent radical associated to the parabolic subgroup in the product $\mathrm{St}(\Delta_1)\times \ldots \times \mathrm{St}(\Delta_k)$. Now the Jacquet module of $\pi'_{N^-}$ is computed in \cite{KL12}. In order to describe the composition factors of $\pi'_{N^-}$ which contribute non-zero Ext take the form, we need some more notations: For a Speh representation $\tau= u_{\rho}(m,d)$, we associate with a collection $\mathcal S_{\rho,m,d}$ of 'hook-shaped multisegments': 
\[\left\{ [\nu^{-(m+d-2)/2}\rho], \ldots, [\nu^{(m-d)/2-1}\rho], [\nu^{(m-d)/2}\rho, \nu^{(m+d-2)/2}\rho]  \right\}  \]
\[ \left\{ [\nu^{-(m+d-2)/2+1}\rho], \ldots, [\nu^{(m-d)/2-1}\rho], [\nu^{(m-d)/2}\rho, \nu^{(m+d-2)/2-1}\rho] \right\}
\]
\[ \ldots
\]
It ends in a segment depending on $m,d$: if $m>d$, the last multisegment takes the form: \[\left\{ [\nu^{-(m-d)/2} \rho], \ldots, [\nu^{(m-d)/2}\rho] \right\} \]
and, if $m<d$, the last multisegment takes the form
\[  \left\{ [ \nu^{-(m-d)/2 }\rho, \nu^{(m-d)/2}\rho ]    \right\}\]
and if $m=d$, the last multisegment takes the form $[\rho]$.

Now we arrange the segments $\Delta_1, \ldots, \Delta_k$ such that if $\Delta_i \cap \Delta_j \neq \emptyset$ and $i<j$, then $\Delta_i \subset \Delta_j$. For a Speh representation $u_{\rho}(m,d)$, define $X(u_{\rho}(m,d))=(m-d)/2$. We shall arrange the Speh representations in 
\[ \pi'=\pi_q^- =\pi_{q,1}^-\times \ldots \times \pi_{q,s}^-
\]
such that $X(\pi_{q,1}^-) \leq \ldots \leq X(\pi_{q,s}^-)$. 

 Using the Kret-Lapid description of Jacquet modules of Speh representations \cite{KL12}, we have the following key properties of $u_{\rho}(m,d)_{N_r^-}$ (for some $r$):
\begin{itemize}
\item $u_{\rho}(m,d)_{N_r^-}$ is semisimple; 
\item for any irreducible composition factor $\omega_1\boxtimes \omega_2$ of $u_{\rho}(m,d)_{N_r^-}$, $\nu^{(m-d)/2}\rho$ is in $\mathrm{cupp}(\omega_2)$.
\end{itemize}
In order to compute (*), we first consider Ext of the form:
\[ (**)\quad  \mathrm{Ext}^i((\mathrm{St}(\Delta_1) \times \ldots \times \mathrm{St}(\Delta_{k-1})\boxtimes \mathrm{St}(\Delta_k), \pi'_{N^-})
\]
and so we have to compute $\pi'_{N^-}$. By geometric lemma, a composition factor takes the form:
\[ \tau \boxtimes (\omega_1 \times \ldots \times \omega_s)
\]
where each $\omega_l$ comes from an irreducible composition factor $\delta \boxtimes \omega_l$ in some Jacquet functor $(\pi_{q,l}^-)_{N^-}$ for some opposite unipotent subgroup $N^-$. We claim the following: \\

\noindent
{\it Claim:} Let $\sigma \in \Delta_k$. Suppose 
\[ (***)\quad  \mathrm{Ext}^i((\mathrm{St}(\Delta_1) \times \ldots \times \mathrm{St}(\Delta_{k-1})\boxtimes \mathrm{St}(\Delta_k), \tau \boxtimes (\omega_1\times \ldots \times \omega_s)) \neq 0 .
\] 
Then $\mathrm{cupp}(\omega_1), \ldots \mathrm{cupp}(\omega_s)$ satisfies a descending pattern, which means that, for any $\nu^c \rho\in \mathrm{cupp}(\omega_x)$ and $\nu^d\rho \in \mathrm{cupp}(\omega_y)$ with $x <y$, we have $c >d$.

\noindent
{\it Proof of Claim:} In order to have (***) to be non-zero, by K\"unneth formula, we must have that $\mathrm{Ext}^{i'}(\mathrm{St}(\Delta_k), \omega_1\times \ldots \times \omega_s) \neq 0$ for some $i'$. Now one applies Frobenius reciprocity, and the corresponding Jacquet functor on $\mathrm{St}(\Delta_k)$ is known (see Section \ref{sec jacquet functor st tri} below). The claim then follows by comparing cuspidal support. \\

Recall that the Speh representations in $\pi'$ is specially arranged, and so the claim with the second bullet of the key properties above implies that there is {\it exactly one} $\omega_l$ is not the trivial representation of $G_0$.

Now we also recall that the Steinberg representations in $\pi^{\mathrm{gen}}_q$ is specially arranged with $\Delta_k$ satisfying certain maximality condition. Such arrangement actually forces that the underlying multisegment of $\omega_l$ is a hook-shaped multisegment, which comes from $\pi_{q,l}$ and the corresponding composition factor in $\pi'_{N^-}$ takes the form $\kappa \boxtimes \omega_l$, where 
\[  \kappa =\pi_{q,1}\times \ldots \times \pi_{q,l-1} \times \widetilde{\pi}_{q,l} \times \pi_{q,l+1} \times \ldots \times \pi_{q,s} ,
\]
and $\widetilde{\pi}_{q,l}$ is a Speh representation such that $\widetilde{\pi}_{q,l}\boxtimes \omega_l$ is a compaction factor of $\pi'_{N^-}$. 

Now, rearranging the Speh representations in $\kappa$ if necessary, one proceeds similarly and inductively for $\Delta_{k-1}, \ldots, \Delta_1$ to find a composition factor in $\pi'_{N^-}$ contributing a non-zero $E_i$.

Hence, the composition factorof $\pi_{N^-}'$ that could contribute a non-zero Ext is isomorphic to objects taking the form
\[ \boxtimes_{\rho,m,d} \boxtimes_{ \mathfrak m \in \cup \mathcal S_{\rho,m,d}} \langle \mathfrak m \rangle ,
\]
where $\rho,m,d$ runs through all the data that $u_{\rho}(m,d)$ is a factor in the Arthur type representation $\pi_N^-$, counting multiplicities.   Now using K\"unneth formula, the computation of $\mathrm{Ext}^i$ follows from 
\[ (\bullet)\quad \bigoplus_{\sum k_{\rho,m,d}=i} \bigotimes_{\rho,m,d} \bigotimes_{\mathfrak m \in \mathcal S_{\rho,m,d}} \mathrm{Ext}^{k_{\rho,m,d}}( \mathrm{St}( \Delta^{gen}(\mathfrak m)), \langle \mathfrak m \rangle) ,
\]
where $\Delta^{gen}(\mathfrak m)$ is the segment with the same cuspidal representations as $\langle \mathfrak m \rangle$, and $\rho,m,d$ run over all data as above. Then when $\rho=1$, it follows from \cite{Or05} that for each set of data $\rho,m,d$, there is at least one $k$ such that 
\[ \mathrm{Ext}^k(\mathrm{St}(\Delta^{gen}(\mathfrak m)), \langle \mathfrak m \rangle) \neq 0
\]
and one can deduce the general case from a transfer argument of Hecke algebra. We pick the smallest such $k$'s and denote the sum of those $k$'s by $k^*$. Such $k^*$ is the smallest integer such that Ext of the following form
\[  \mathrm{Ext}^{k^*}(\bigotimes_{\rho,m,d} \bigotimes_{\mathfrak m \in \mathcal S_{\rho,m,d}} \mathrm{St}(\Delta^{gen}(\mathfrak m), \bigotimes_{\rho,m,d} \bigotimes_{\mathfrak m \in \mathcal S_{\rho,m,d}} \langle \mathfrak m \rangle  )
\]
is non-zero. The hook-shaped multisegments obtained above (see (**) and (***)) come from all Speh representations $\left\{ \pi_{q,a} \right\}_{a=1}^s$ possibly in different orders, but any simple composition factors in $\pi'_{N^-}$ obtained above will still give the same ($\bullet$) after K\"unneth formula. Thus, a long exact sequence argument can conclude that $E_{k^*} \neq 0$.

(3) is similar to (2). We omit the details.
\end{proof}

\begin{remark}

For some other related computations of Arthur type representations, for example, see Ext-groups of tempered representations \cite{OS12} and Speh representations from Koszul resolution \cite{Ch16}.
\end{remark}

\subsection{Another example}

One can obtain different information from various filtrations on restricted representations \cite{Pr93, CS18, Ch19} such as left and right Bernstein-Zelevinsky filtrations \cite{CS18, Ch19}. We shall see another example below using combinations of filtrations:

\begin{example} \label{ex rank 2 comput}
Let $\Delta[d]=[\nu^{-(d-1)/2}, \nu^{(d-1)/2}]$. For $e \geq 3$, let 
\[ \pi_1 =  \langle \Delta[e] \rangle \times \mathrm{St}(\Delta[e-2])  \times \sigma  ,\] 
and let 
\[ \pi_2= \mathrm{St}(\Delta[e-1]) \times \langle \Delta[e-1]  \rangle , \]
where $\sigma$ is a ramified character.

We first investigate possible Bernstein-Zelevinsky layers contributing non-zero Ext-groups. Consider the derivatives:
\[  {}^{(i_1)}\langle \Delta[e] \rangle \times {}^{(i_2)}\mathrm{St}( \Delta[e-2])  \times {}^{(i_3)}\sigma \quad \mbox{ and }  \mathrm{St}( \Delta[e-1])^{(j_1)}\times \langle \Delta[e-1]\rangle^{(j_2)}  
\]
and, by comparing cuspidal supports, we must have $i_1=1$.
Then we have the following two possibilities: either
\begin{enumerate}
\item $j_1=e-1$; or
\item $j_2=1$; or 
\end{enumerate}
In the case that $j_1=e-1$, by comparing cuspidal support, we have $j_2=0$, and then $i_2=e-2$. In the case $j_2=1$, we have two possibilities:
\begin{enumerate}
\item $j_1=0$, $i_2=0$.
\item $j_1=e-2$, $i_2=e-2$
\end{enumerate}

Now we find a cuspidal representation $\sigma'$ as in Proposition \ref{prop bessel transfer} to consider the representation $\pi_2\times \sigma'$. Now we observe that there is two layers $(\pi_2\times \sigma')|_M$ that contribute non-zero Ext-groups (after restricting to $G$): Now $(j_1, j_2)=(e-1,0)$, it contributes one layer
\[ \lambda_1:=\langle \Delta[e-1] \rangle \times \Pi_{e+1}
\]
and $(j_1,j_2)=(0,1)$, it contributes one layer
\[\lambda_2 := \mathrm{St}(\Delta[e-1]) \times \langle \nu^{-1/2}\Delta[e-2]\rangle  \times \Pi_3 
\]
and $(j_1,j_2)=(e-2,1)$, it contributes one (reducible) layer
\[ \lambda_3: =\lambda = \langle \nu^{-1/2} \Delta[e-2]\rangle \times \nu^{(e-1)/2}  \times \Pi_{e+1}.
\]
We remark that $\lambda_3$ is indecomposable as $\langle \nu^{1/2}\Delta[e-2] \rangle \times \nu^{-(e-1)/2}$ is indecomposable.

We now consider the dual restriction problem in Proposition \ref{prop bessel transfer}, and so we consider the restriction for $\pi_2 \times \sigma'$ for some cuspidal representation $\sigma'$ of $G_2$.

Using the following short exact sequence (Lemma \ref{lem short exact seq}):
\[   0 \rightarrow  \langle \Delta[e-1] \rangle|_M \bar{\times} (\mathrm{St}(\Delta[e-1]) \times \sigma')    \rightarrow (\pi_2 \times \sigma')|_M    \rightarrow    \langle \Delta[e-1] \rangle \bar{\times} ((\mathrm{St}(\Delta[e-1]) \bar{\times}\sigma')|_M)      \rightarrow 0 , \]
and letting 
\[  X^* = \langle \Delta[e-1] \rangle|_M \bar{\times} (\mathrm{St}(\Delta[e-1]) \times \sigma') ,
\]
 $X^*$ admits a filtration, in which there is one successive quotient isomorphic to $\lambda_2$ and another successive quotient isomorphic to $\lambda_3$.

Using Bernstein-Zelevinsky filtration, we obtain a filtration on $(\pi_2 \times \sigma')|_M$ of the form
\[  0 =Y_{2e} \subset Y_{2e-1} \subset \ldots \subset Y_0 =(\pi_2 \times \sigma')|_M .
\]
so that
\begin{enumerate}
\item $Y_{e}/Y_{e+1} \cong  (\pi_2 \times \sigma')^{(e+1)} \bar{\times} \Pi_{e+1}$, and 
\item $Y_{e+1}$ is a simple module which is not isomorphic to any simple composition factor of $\lambda_1,\lambda_2, \lambda_3$, and
\item $Y_{e}/Y_{e+1}$ admits a filtration with one quotient isomorphic to $\lambda_1$ and another quotient isomorphic to $\lambda_3$.
\end{enumerate}

The key of two filtrations is to obtain the following filtration, as $M_{n+2}$, and the direct sum in the quotient roughly contributes the direct sum of Ext-groups in Conjecture \ref{conj ext sum}:
\[  0 \rightarrow I \rightarrow X^*+Y_e \rightarrow X^*/I \oplus Y_e/I \rightarrow 0 ,
\]
where $I=X^* \cap Y_e$. Let
\[ \beta:=\langle \left\{ \nu^{-1/2}\Delta[e-2] , \nu^{(e-1)/2} \right\} \rangle \times \Pi_{e+1},
\]
which has multiplicity one in $\pi_2\times \sigma'|_M$. With the above information on $X^*$ and $Y_e$, we can obtain further structure on $I$. The multiplicity forces that $I$ contains the unique composition factor $\beta$, but the indecomposability of $\lambda_3$ also forces $I$ contains the composition factor $\beta$, and a count on multiplicities gives that other composition factor of $I$ is not isomorphic to $\lambda_1, \lambda_2$ or $\beta$ (those are all the possible factors contributing non-zero Ext). Thus, we have that, for all $k$,
\[ \mathrm{Ext}_{G_{n+1}}^k(I|_{G_{n+1}}, \pi_1) =\mathrm{Ext}_{G_{n+1}}^k(\lambda_3|_{G_{n+1}}, \pi_1)=0 .
\]

Then we have that
\begin{align*}
  \mathrm{Ext}^k_{G_{n+1}}(\pi_2 \times \sigma', \pi_1)
 \cong & \mathrm{Ext}^k_{G_{n+1}}((X^*+Y_{e})|_{G_{n+1}}, \pi_1)  \\
 \cong & \mathrm{Ext}^k_{G_{n+1}}(X^*/I, \pi_1)\oplus  \mathrm{Ext}^k_{G_{n+1}}(Y_{e}/I, \pi_1)  	\\
\cong & \mathrm{Ext}^k_{G_{n+1}}(\lambda_2, \pi_1) \oplus \mathrm{Ext}^k_{G_{n+1}}(\lambda_1, \pi_1) \\
\cong & \mathrm{Ext}^k_{G_{n-1}}((\pi_2^{[1]}, {}^{(2)} \pi_1  ) \oplus \mathrm{Ext}^k_{G_{n+1-e}}(\pi_2^{[e-1]}, {}^{(e)}\pi_1)
\end{align*}

The first isomorphism follows from that the quotients by $X^*+Y_e$ has zero Ext by looking at the possible composition factors and some computations on comparing cuspidal supports. The fourth isomorphism follows from the adjointness of the functors (see \cite[Lemma 2.1]{CS18} for more discussions).

Since $\pi_1^{\vee} \cong \pi_1$ and $\pi_2^{\vee} \cong \pi_2$, taking duals and using Proposition \ref{prop bessel transfer} gives that
\[\mathrm{Ext}^k_{G_{n}}(\pi_1, \pi_2) \cong \mathrm{Ext}^k_{G_{n-1}}(\pi_1^{[2]}, {}^{(1)} \pi_2  ) \oplus \mathrm{Ext}^k_{G_{n+1-e}}(\pi_1^{[e]}, {}^{(e-1)}\pi_2) .
\]
The last isomorphism follows from \cite[Lemma 2.2]{CS18}.

\end{example}




\section{Product preserving extensions} \label{ss prod preserve}

A motivating example in this section is the following. Let $\sigma$ be an irreducible cuspidal representation of $G_n$. Let $\pi_1$ and $\pi_2$ be two admissible representations of $G_k$ such that the cuspidal supports of irreducible composition factors of $\pi_1$ and $\pi_2$ do not contain $\sigma$. Then, a simple application of Frobenius reciprocity and geometric lemma gives that
\[  \mathrm{Hom}_{G_{n+k}}(\sigma \times \pi_1, \sigma \times \pi_2) \cong \mathrm{Hom}_{G_n}(\sigma, \sigma) \boxtimes \mathrm{Hom}_{G_k}(\pi_1, \pi_2) \cong \mathrm{Hom}_{G_k}(\pi_1, \pi_2) .
\]
Our goal is to generalize the above isomorphism to  a larger class of examples in a functorial way, which is Theorem \ref{thm fully faith product}.

\subsection{Preserving extensions}

Let $\mathcal C \subset \mathrm{Irr}^c$. Define $\mathrm{Alg}_{\mathcal C}(G_m)$ to be the full subcategory of $\mathrm{Alg}(G_m)$ whose objects $\pi$ have finite lengths and satisfy the property that for any simple composition factor $\pi'$ of $\pi$, and for any $\sigma \in \mathrm{cupp}(\pi')$, $\sigma$ lies in $\mathcal C$.

\begin{theorem}  \label{thm preserve extension}
Let $\rho \in \mathrm{Irr}^{c}(G_k)$. Let $\mathcal C=\mathcal C_{u_{\rho}(d,m)}$.
Let $\pi \in \mathrm{Alg}_{\mathcal C}(G_n)$ with length $2$. Then $\pi$ is indecomposable if and only if $u_{\rho}(m,d) \times \pi$ is indecomposable.

\end{theorem}


We will prove Theorem \ref{thm preserve extension} in Section \ref{ss proof of preserve ext}. We expect to prove a more general result elsewhere using some ideas from \cite{Ch18} (see Section \ref{ss generalize}) as well as the case mentioned here.


The idea of the proof is to first prove for a large Speh representation (in the sense of Section \ref{ss fully-faith large Speh}). In such case, one can compute via some simpler computations of Jacquet modules and standard modules. The general case is deduced from 'truncating' large Speh representations to the desired one.

\begin{remark}

In general, a product does not preserve extensions even if it preserves irreducibility of composition factors. The standard example is that $\nu \times (1 \times \nu)$, which is of length $2$. In this case, $\nu \times \langle [1,\nu] \rangle$ and $\nu \times \mathrm{St}([1,\nu])$ are both irreducible, but $1\times \nu$ is indecomposable and $\nu\times (1 \times \nu)$ is semisimple.
\end{remark}





\subsection{Jacquet functors} \label{sec jacquet functor st tri}
Recall that $N_p$ is the subgroup of $G_n$ containing all matrices $\begin{pmatrix} I_{n-p} & u \\ & I_p \end{pmatrix}$, where $u \in Mat_{n-p,p}$. 

Let $\Delta=[\nu^a \rho, \nu^b\rho]$ be a Zelevinsky segment. Let $m=n_{\rho}$. Then \cite[Propositions 3.4 and 9.5]{Ze80}, the Jacquet modules are:
\[  \langle \Delta \rangle_{N_{mi}}=[\nu^a\rho, \nu^{b-i}\rho] \boxtimes [\nu^{b-i+1}\rho, \nu^b \rho] .
\]
\[  \langle \Delta \rangle_{N_{mi}^-}=\langle [\nu^{a+i}\rho, \nu^b \rho]  \rangle \boxtimes \langle [\nu^a\rho, \nu^{a+i-1}\rho] \rangle
\]
\[  \mathrm{St}(\Delta)_{N_{mi}} =\mathrm{St}([\nu^{a+i}\rho, \nu^b\rho ]) \boxtimes \mathrm{St}([\nu^a\rho, \nu^{a+i-1}\rho])
\]
\[ \mathrm{St}(\Delta)_{N_{mi}^-} =\mathrm{St}([\nu^a\rho, \nu^{b-i}\rho] \boxtimes \mathrm{St}([\nu^{b-i+1}\rho, \nu^b\rho]) .
\]

Note that computing $\pi_{N_{i}^-}$ is equivalent to first computing $\pi_{N_{n-i}}$ to obtain a $G_{i}\times G_{n-i}$-representation, then twisting by the action by the element $\begin{pmatrix} 0 & I_i \\ I_{n-i} & 0 \end{pmatrix}$ to obtain a $G_{n-i} \times G_{i}$-representation.

\subsection{Fully-faith product for large Speh} \label{ss fully-faith large Speh}


For $\rho \in \mathrm{Irr}^c(G_k)$, $d,m \in \mathbb{Z}_{\geq 1}$, let $\widetilde{\Delta}_{\rho}(d,k)=[\nu^{-(d-1)/2}\rho, \nu^{(d-1)/2+k}\rho ]$. We first consider 
\[ \widetilde{\mathfrak m}_{\rho}(m,d,k) =\left\{ \nu^{-(m-1)/2}\widetilde{\Delta}(d,k), \ldots , \nu^{(m-1)/2}\widetilde{\Delta}(d,k) \right\} .
\]
Let $\widetilde{u}_{\rho}(m,d,k)=\langle \widetilde{\mathfrak m}_{\rho}(m,d,k) \rangle$, which is sometimes called essentially Speh representation as it is a Speh representation twisted by a character. In particular, $\widetilde{u}_{\rho}(m,d,0)=u_{\rho}(m,d)$ if $\rho$ is unitarizable. 


\begin{lemma} \label{lem iso big}
Let $\pi_1, \pi_2$ be admissible representations of $G_n$. Fix $\rho \in \mathrm{Irr}^c$ and $d,m \in \mathbb{Z}_{\geq 1}$. For any $k \geq 0$, set $\widetilde{u}_k=\widetilde{u}_{\rho}(m,d,k)$. For $k$ large enough, we have a natural isomorphism:
\[  \mathrm{Hom}_{G_n}(\pi_1, \pi_2) \cong \mathrm{Hom}_{G_{n+p}}(\widetilde{u}_k \times \pi_1,  \widetilde{u}_k \times \pi_2) ,
\]
where $p=n_{\rho}m(d+k)$. Here naturality holds, when the isomorphism holds, for both $\pi_1$ and $\pi_2$ for the same $k$.
\end{lemma}

\begin{proof}
We set $k$ large enough such that $\nu^{(d-m)/2+k}\rho$ is not in the cuspidal supports of any irreducible representation of $\pi_1$ and $\pi_2$. 

Let $\mathfrak m=\widetilde{\mathfrak m}_{\rho}(m,d,k)$ and let $\widetilde u=\widetilde{u}_{\rho}(m,d,k)$. Using the injection:
\[ \widetilde u \times \pi_2= \langle \mathfrak m \rangle \times \pi_2 \hookrightarrow \zeta(\mathfrak m) \times \pi_2 ,
\]
the left exactness of $\mathrm{Hom}_{G_{n+p}}(\widetilde u \times \pi_1,.)$ gives
\begin{align} \label{eqn inject first}
  \mathrm{Hom}_{G_{n+p}}(\widetilde u \times \pi_1 ,\zeta(\mathfrak m) \times \pi_2) 
\hookleftarrow  & \mathrm{Hom}_{G_{n+p}} (\widetilde u \times \pi_1 ,  \widetilde u \times \pi_2 )
\end{align}

Let $\Delta=[\nu^{(-d+m)/2}\rho  , \nu^{(d+m-2)/2+k}\rho]$. Since $\zeta(\mathfrak m)=\langle \Delta \rangle \times \zeta(\mathfrak m\setminus \left\{ \Delta \right\})$, 
\[ \mathrm{Hom}_{G_{n+p}}(\langle \mathfrak m \rangle \times \pi_1, \zeta(\mathfrak m)\times \pi_2) \cong \mathrm{Hom}_{G_{n+p}}(\langle \mathfrak m \rangle \times \pi_1, \langle \Delta \rangle \times \pi' ),
\]
where $\pi'=\zeta( \mathfrak m \setminus \left\{\Delta \right\}) \times \pi_2$. 

Let $q=n_{\rho}m$. Now Frobenius reciprocity gives that
\[ \mathrm{Hom}_{G_{n+p}}(\langle \mathfrak m \rangle \times \pi_1, \langle \Delta \rangle \times \pi' )\cong \mathrm{Hom}_{G_q\times G_{n+p-q}}((\langle \mathfrak m \rangle \times \pi_1)_{N_{n+p-q}}, \langle \Delta \rangle \boxtimes \pi' ) .
\]
Note that $\nu^{(d+m-2)/2+k}\rho$ does not appear in the cuspidal support of irreducible factors of $\pi_1$. With some analysis on Jacquet module from the geometric lemma (see, for example the proof of Lemma \ref{lem taking jacquet form as st} below for more details), the only composition factor in $(\langle \mathfrak m \rangle \times \pi_1)_{N_{n+p-q}}$ that has the same cuspidal support as $\langle \Delta \rangle \boxtimes \pi'$ is 
\[   \langle \Delta \rangle \boxtimes \langle\mathfrak m \setminus \left\{ \Delta \right\} \rangle \times \pi_1 . 
\]
Thus we have
\[ \mathrm{Hom}(\langle \mathfrak m \rangle \times \pi_1, \langle \Delta \rangle \times \pi' )\cong \mathrm{Hom}(\langle \mathfrak m\setminus \left\{ \Delta \right\} \rangle \times \pi_1, \pi')=\mathrm{Hom}(\langle \mathfrak m' \rangle \times \pi_1, \zeta(\mathfrak m')\times \pi_2)  ,
\]
where $\mathfrak m'=\mathfrak m \setminus \left\{ \Delta \right\}$, and so
\[\mathrm{Hom}(\langle \mathfrak m \rangle \times \pi_1,\zeta(\mathfrak m)\times \pi_2)=\mathrm{Hom}(\langle \mathfrak m' \rangle \times \pi_1, \zeta(\mathfrak m')\times \pi_2)
\]  

Since $\nu^{(d+m-2)/2+k-1}\rho$ does not appear in the cuspidal support of $\pi'$ (when $k \geq 2$, otherwise we are done), we can repeat the similar process by replacing $\mathfrak m\setminus \left\{ \Delta \right\}$ with $\mathfrak m$. Inductively (which works by our choice of large $k$), we obtain
\[  \mathrm{Hom}_{G_{n+p}}(\langle \mathfrak m \rangle \times \pi_1, \zeta(\mathfrak m) \times \pi_2) \cong \mathrm{Hom}_{G_n}(\pi_1, \pi_2) 
\]
With (\ref{eqn inject first}),
\begin{align} \label{eqn inj second}
 \mathrm{Hom}_{G_n}(\pi_1, \pi_2) \hookleftarrow \mathrm{Hom}_{G_{n+p}} (\widetilde u \times \pi_1 ,  \widetilde u \times \pi_2 ) .
\end{align}

Viewing $\widetilde{u} \times $ as a functor and using the faithfulness of $\widetilde{u}\times$ (see Section \ref{ss product functor} below), we have that 
\begin{align} \label{eqn inj third} \mathrm{Hom}_{G_n}(\pi_1, \pi_2) \hookrightarrow \mathrm{Hom}_{G_{n+p}}( \widetilde{u}  \times \pi_1,  \widetilde{u}  \times \pi_2)
\end{align}
Since we are dealing with admissible representations, the injections in (\ref{eqn inj second}) and (\ref{eqn inj third}) must be isomorphisms. Hence, we have that:
\[  \mathrm{Hom}_{G_n}(\pi_1, \pi_2) \cong \mathrm{Hom}_{G_{n+p}}( \widetilde u \times \pi_1, \widetilde u \times \pi_2) . 
\]
\end{proof}


\begin{remark}
We remark that the above lemma does not require $\pi_1$ and $\pi_2$ to be in $\mathrm{Alg}_{\mathcal C}(G_n)$. In such case, $\widetilde{u}_{\rho}(m,d,k) \times \pi_1$ may have more complicated structure. For example, when $\pi_1$ has unique quotient, the cosocle of $\widetilde{u}_{\rho}(m,d,k) \times \pi$ may not be irreducible. We give an example here.

Let $\Delta=[\nu^{1/2}, \nu^k]$ for sufficiently large $k$. Let $\pi=\nu^{-1/2} \times \nu^{1/2}$, which is reducible with length $2$. Then 
\[  \langle \Delta \rangle \times \pi\]
has the quotient $\langle [\nu^{-1/2},\nu^k] \rangle \times \nu^{1/2}$ since $\langle \Delta \rangle \times \nu^{-1/2}$ has quotient $\langle [\nu^{-1/2},\nu^k] \rangle$, and has the quotient $\langle \Delta \rangle \times \mathrm{St}([\nu^{-1/2},\nu^{1/2}])$, which is irreducible (deduced from similar way as in \cite[Appendix]{Ch19}), since $\pi$ has the quotient $\mathrm{St}([\nu^{-1/2}, \nu^{1/2}])$.
\end{remark}

\subsection{Product for irreducibility}

We use the notations in the previous section.
\begin{lemma} \label{lem speh irred} \cite{LM16}
Fix $m,d$ and $\rho \in \mathrm{Irr}^c$. Let $\mathfrak m_1$ and $\mathfrak m_2$ be multisegments with each segment $\Delta$ satisfying that any cuspdial representation in $\Delta$ is in $\mathcal C_{u_{\rho}(m,d)}$. 
Then, for any $k \geq 0$,
\begin{enumerate}
\item $\widetilde{u}_{\rho}(m,d,k) \times \langle \mathfrak m_i \rangle$ ($i=1,2$) is irreducible;
\item $\widetilde{u}_{\rho}(m,d,k) \times \langle \mathfrak m_1 \rangle \cong \widetilde{u}_{\rho}(m,d,k) \times \langle \mathfrak m_2 \rangle$ if and only if $\mathfrak m_1 =\mathfrak m_2$;
\item $\widetilde{u}_{\rho}(m,d,k) \times \langle \mathfrak m_i \rangle \cong \langle \mathfrak m_i \rangle \times \widetilde{u}_{\rho}(m,d,k)$, for $i=1,2$;
\item Suppose $\omega$ be an irreducible representation of $G_{a+p}$. If $\widetilde{u}_{\rho}(m,d,k) \boxtimes \pi$ is an irreducible quotient of $\omega_{N}$, then $\omega \cong \widetilde{u}_{\rho}(m,d,k) \times \pi$. The statement also holds if we replace $\omega_N$ by $\omega_{N^-}$ as well as replace quotient by submodule. 
\end{enumerate} 
\end{lemma}

\begin{proof}
(1) and (2) follow from \cite[Corollary 6.7]{LM16}. We only sketch how to deduce from \cite[Appendix]{Ch19}. Using a modified version of a lemma in \cite[Appendix]{Ch19}, we have that 
\[  \theta(\zeta(\widetilde{\mathfrak m}_{\rho}(m,d,k)+\mathfrak m_i))^{\vee} \twoheadrightarrow u_{\rho}(m,d,k) \times \langle \mathfrak m_i \rangle  \hookrightarrow \zeta(\widetilde{\mathfrak m}_{\rho}(m,d,k)+\mathfrak m_i) ,\]
which forces that $\widetilde{u}_{\rho}(m,d,k)\times \langle \mathfrak m_i \rangle$ is the unique submodule of $\zeta(\widetilde{\mathfrak m}_{\rho}(m,d,k)+\mathfrak m_i)$.  (4) follows from Frobenius reciprocity and (1). (3) follows from the Gelfand-Kazhdan involution.

\end{proof}

\subsection{Proof of Theorem \ref{thm preserve extension}} \label{ss proof of preserve ext}



We fix $\rho, d, m$. For simplicity, set $\widetilde{u}_k=u_{\rho}(m,d,k)$ for $k \geq 0$. Let $\Delta_{k+1}=[\nu^{{(m-d)/2+k+1}} \rho, \nu^{(m+d-2)/2+k+1}\rho]$. Let $\mathcal C$ be as in Theorem \ref{thm preserve extension} for such $\rho$, $d$ and $m$. 

\begin{lemma} \label{lem taking jacquet form as st}
 Let $p=n_{\rho}m$. Let $\pi'$ be an irreducible representation in $\mathrm{Alg}_{\mathcal C}(G_{n'})$. Let $n=n'+(d+k+1)mn_{\rho}$. There is a unique irreducible composition factor $\omega$ in
\[(\mathrm{St}(\Delta_{k+1}) \times \widetilde{u}_{k} \times \pi')_{N_{n-p}^-}\] which is isomorphic to $\mathrm{St}(\Delta_{k+1}) \boxtimes \tau$ for some irreducible $\tau$ of $G_{n-p}$, and moreover, 
\[\omega \cong \mathrm{St}(\Delta_{k+1}) \boxtimes (\widetilde{u}_k\times \pi'). \] 
\end{lemma}

\begin{proof}

For simplicity, set $\lambda=\widetilde{u}_k \times \pi'$, which is irreducible by Lemma \ref{lem speh irred}. Note that $\nu^{(m+d-2)/2+k+1}\rho$ is not in the cuspidal support of $\widetilde{u}_{k}\times \pi'$. To compute $(\mathrm{St}(\Delta_{k+1}) \times \lambda)_{N_{n-p}^-}$, we first compute
\[  (  \mathrm{St}(\Delta_{k+1}) \times \lambda )_{N_p}
\]
(see discussions in Section \ref{sec jacquet functor st tri}), and then twisting the action by an element. Then geometric lemma on $( \mathrm{St}(\Delta_{k+1}) \times \lambda)_{N_p}$ yields a filtration successive quotients of the form 
\[   \mathrm{St}([\nu^{l+1}\rho, \nu^b\rho])\times \omega  \boxtimes  \mathrm{St}([\nu^a\rho, \nu^l]) \times \omega'  .
\]
and this gives a filtration on $(\widetilde{u}_k \times \mathrm{St}(\Delta_{k+1}))_{N_{n-p}^-}$ with successive quotients taking the form
\begin{align}\label{eqn form gl}   \mathrm{St}([\nu^a\rho, \nu^{l}\rho])  \times \omega' \boxtimes \mathrm{St}(\nu^{l+1}\rho, \nu^b\rho]) \times \omega .
\end{align}
Here $\omega$ and $\omega'$ are representations whose cuspidal supports do not contain $\nu^{(m+d-2)/2+k+1}\rho$. Thus an irreducible composition factor $\gamma$ of $(\mathrm{St}(\Delta_{k+1}) \times \lambda )_{N_{n-p}^-}$ can take the form $\mathrm{St}(\Delta_{k+1})\boxtimes \tau$ only if $l =b$ in (\ref{eqn form gl}). In such case, the successive quotient from geometric lemma is irreducible and is isomorphic to $\gamma\cong \mathrm{St}(\Delta_{k+1})\boxtimes \lambda$. 
\end{proof}

\begin{lemma} \label{lem surjection steinberg rep}
There exists a surjection from $\mathrm{St}(\Delta_{k+1}) \times \widetilde{u}_k$ to $\widetilde{u}_{k+1}$.
\end{lemma}

\begin{proof}
Let $\Delta=\Delta_{k+1}$. It follows from Lemma \ref{lem st ze transfer} that there is a surjection
\[   \tau:= \mathrm{St}(\Delta) \times \mathrm{St}(\nu^{-1}\Delta) \times \ldots \times \mathrm{St}(\nu^{-(d+k)}\Delta) \rightarrow \widetilde{u}_{k+1},
\]
and similarly, $\tau' :=  \mathrm{St}(\nu^{-1}\Delta) \times \ldots \times \mathrm{St}(\nu^{-(d+k)}\Delta) \rightarrow \widetilde{u}_{k}$. By uniqueness of the irreducible quotient for $\tau$, we then also have that $\mathrm{St}(\Delta) \times \widetilde{u}_k$ has the same unique irreducible quotient as $\tau$. This gives surjections
\[  \tau = \mathrm{St}(\Delta)\times \tau' \twoheadrightarrow \mathrm{St}(\Delta) \times \widetilde{u}_k \twoheadrightarrow \widetilde{u}_{k+1}.
\]
\end{proof}

\begin{lemma} \label{lem zero hom kernel}
Let $K$ be the kernel of the surjection in Lemma \ref{lem surjection steinberg rep}. For any $\pi$ in $\mathrm{Alg}_{\mathcal C}(G_{n'})$ and any $\pi'$ in $\mathrm{Alg}_{\mathcal C}(G_{n'})$, 
\[  \mathrm{Hom}(K \times \pi, \widetilde{u}_{k+1} \times \pi') =0 .
\]
\end{lemma}

\begin{proof}
Let $\Delta=\Delta_{k+1}$. We have the following short exact sequence:
\[  0 \rightarrow K \rightarrow \mathrm{St}(\Delta) \times \widetilde{u}_k \rightarrow \widetilde{u}_{k+1} \rightarrow 0 ,
\]
which gives the short exact sequence:
\[ 0 \rightarrow K \times \pi \rightarrow \mathrm{St}(\Delta) \times \widetilde{u}_k \times \pi \rightarrow \widetilde{u}_{k+1} \times \pi \rightarrow 0 .
\]
Let $N^-=N^-_{n'+n_{\rho}m(d+k)}$. The Jacquet functor is exact and so we have another short exact sequence:
\begin{align} \label{eqn exact after jacquet }
  0 \rightarrow (K \times \pi)_{N^-} \rightarrow (\mathrm{St}(\Delta) \times \widetilde{u}_k \times \pi)_{N^-}  \rightarrow (\widetilde{u}_{k+1}\times \pi)_{N^-}  \rightarrow 0 .
\end{align}
Now,  by second adjointness of Frobenius reciprocity, we have a map 
\[   \mathrm{St}(\Delta) \boxtimes (\widetilde{u}_k \times \pi) \rightarrow (\widetilde{u}_{k+1}\times \pi)_{N^-} .
\]
The map is indeed injective. This follows first from the case that $\pi$ is irreducible by using irreducibility of $\widetilde{u}_k\times \pi$ (Lemma \ref{lem speh irred}), and then lift to the general case by an inductive argument using functoriality of Frobenius reciprocity. (One can also prove the map is injective by directly computing the composition factors of $(\widetilde{u}_{k+1}\times \pi)_{N^-}$ taking the form $\mathrm{St}(\Delta)\boxtimes \tau$, see the proof of Lemma \ref{lem taking jacquet form as st}.)

 Now by Lemma \ref{lem taking jacquet form as st} and counting on composition factors, all irreducible composition factors of the form $\mathrm{St}(\Delta)\boxtimes \tau$ in $(\mathrm{St}(\Delta) \times \widetilde{u}_k \times \pi)_{N^-}$ are mapped onto $(\widetilde{u}_{k+1}\times \pi)_{N^-}$ under the surjection map in (\ref{eqn exact after jacquet }). 

Thus there is no irreducible composition factor of $(K\times \pi)_{N^-}$ taking the form $\mathrm{St}(\Delta)\boxtimes \tau$. On the other hand, for any irreducible $\pi'$, $(\widetilde{u}_{k+1}\times \pi')_{N^-}$ has irreducible composition factor of the form $\mathrm{St}(\Delta) \boxtimes \tau$, which can be deduced by an argument using Frobenius reciprocity. Hence, following from the exactness of Jacquet functor (and Lemma \ref{lem speh irred} (1)), we must have
\[ \mathrm{Hom}(K \times \pi, \widetilde{u}_{k+1}\times \pi') =0 .
\]
\end{proof}

\noindent
{\it Proof of Theorem \ref{thm preserve extension}.} 
We keep using the above notations. Let $\pi \in \mathrm{Alg}_{\mathcal C}(G_n)$ of length $2$. The if direction is easy and so we now consider the only if direction. Suppose $\pi$ is indecomposable. We shall use backward induction to prove that, for any $k \geq 0$, $\widetilde{u}_k \times \pi$ is indecomposable, and moreover $\widetilde{u}_k \times \pi$ has unique irreducible quotient. When $k$ is sufficiently large, Lemma \ref{lem speh irred} implies that $\widetilde{u}_k\times \pi$ has length 2, and Lemma \ref{lem iso big} (and Lemma \ref{lem speh irred} (2)) imply the uniqueness of the quotient, which also then implies the indecomposability.



Let $\pi_1$ and $\pi_2$ be the two irreducible composition factors of $\pi$. Let $\lambda_i=\widetilde{u}_k \times \pi_i$ ($i=1,2$). $\lambda_1$ and $\lambda_2$ are irreducible, and $\pi_1 \cong \pi_2$ $\Leftrightarrow$ $\lambda_1 \cong \lambda_2$  by Lemma \ref{lem speh irred}. 

Suppose $\widetilde{u}_{k} \times \pi$ is not indecomposable. Let $\Delta=\Delta_{k+1}$. This gives an isomorphism 
\[           \widetilde{u}_{k}\times \pi \cong \lambda_1 \oplus \lambda_2 .
\]
and so there exists surjections, by Lemma \ref{lem surjection steinberg rep},
\[  \mathrm{St}(\Delta) \times \widetilde{u}_{k}\times \pi \cong \mathrm{St}(\Delta) \times \lambda_1 \oplus \mathrm{St}(\Delta)\times \lambda_2 \rightarrow \widetilde{u}_{k+1} \times \pi_1 \oplus \widetilde{u}_{k+1} \times \pi_2
\]
This implies that:
\begin{enumerate}
\item if $\lambda_1 \not\cong \lambda_2$, then for {\it both} $i=1,2$, 
\[ \mathrm{Hom}_G( \mathrm{St}(\Delta) \times \widetilde{u}_{k}\times \pi, \widetilde{u}_{k+1}\times \pi_i) \neq 0 ;\]
\item if $\lambda_1 \cong \lambda_2$, then 
\[ \mathrm{dim}~\mathrm{Hom}_G(\mathrm{St}(\Delta) \times \widetilde{u}_k \times \pi, \widetilde{u}_{k+1}\times \pi_1) \geq 2 .
\]
\end{enumerate}


On the other hand, we have the following short exact sequence from Lemma \ref{lem surjection steinberg rep}:
\[  0 \rightarrow K \times \pi \rightarrow       \mathrm{St}(\Delta) \times \widetilde{u}_{k} \times \pi \rightarrow \widetilde{u}_{k+1} \times \pi \rightarrow 0 .
\]
By Lemma \ref{lem zero hom kernel}, $\mathrm{Hom}(K\times \pi, \widetilde{u}_{k+1} \times \pi_i) =0$ for $i=1,2$. Hence we have
\[ \mathrm{Hom}( \widetilde{u}_{k+1}\times \pi, \widetilde{u}_{k+1}  \times \pi_i)\cong  \mathrm{Hom}(\mathrm{St}(\Delta)\times \widetilde{u}_{k} \times \pi, \widetilde{u}_{k+1} \times \pi_i) .
\]
However, by induction hypothesis and irreducibility of $\widetilde{u}_{k+1}\times \pi_i$, the former Hom has dimension one for both $i=1$ or $2$ if $\lambda_1 \cong \lambda_2$, and has dimension one for precisely one of $i=1,2$ if $\lambda_1 \not\cong \lambda_2$. This gives a contradiction to (1) or (2) above. Thus $\widetilde{u}_{k}\times \pi$ is indecomposable as desired, and since $\widetilde{u}_k\times \pi$ has length $2$, it also has unique irreducible quotient. This completes the proof.




\section{Product functor of a Speh representation} \label{s product functor}




\subsection{Fully-faithful product} \label{ss product functor}

For an irreducible cuspidal representation $\rho$ of some $G_k$, define $\mathrm{cupp}_{\mathbb Z}(\rho)=\left\{ \nu^n\rho\right\}_{n\in \mathbb{Z}}$. 

Let $\pi \in \mathrm{Alg}_{\mathcal C}(G_p)$. Define the functor 
\[ \times_{\pi, \mathcal C}=\times_{\pi, \mathcal C, n}: \mathrm{Alg}_{\mathcal C}(G_n) \rightarrow \mathrm{Alg}_{\mathcal C}(G_{n+p})\]
 as:
\[   \times_{\pi,\mathcal C}(\omega) = \pi\times \omega ,
\]
and, for a map $\Omega: \omega_1 \rightarrow \omega_2$ in $\mathrm{Alg}_{\mathcal C}(G_n)$, 
\[\times_{\pi, \mathcal C}(\Omega)(f)(g)= (\mathrm{Id}_{\pi} \boxtimes \Omega)(f(g)) ,\]
where $f \in u_{\rho}(m,d) \times\omega_1$ is a smooth function $f: G_{n+p}\rightarrow u_{\rho}(m,d)\boxtimes \omega_1$ (Section \ref{ss para induction}). Note that since $\times_{\pi, \mathcal C}$ is exact and sends a non-zero object to a non-zero object, $\times_{\pi, \mathcal C}$ is faithful. We may sometimes simply write $\times_{\pi}$ for $\times_{\pi, \mathcal C}$.

For an irreducible representation $\pi$, we define a stable cuspidal set $\mathcal C_{\pi}$ of $\pi$ as
\begin{align} \label{eqn stable cuspidal set}
 \mathcal C_{\pi} = \mathrm{cupp}(\pi) \cup (\mathrm{Irr}^c \setminus \mathrm{cupp}_{\mathbb{Z}}(\pi)) .
\end{align}
(Here we regard $\mathrm{cupp}(\pi)$ as a set.) A motivation for the term stable cuspidal set is in the case that for $\pi =u_{\rho}(d,m)$, and for any $\rho \in \mathcal C_{\pi}$, $\rho \times \pi$ is irreducible. (However, this is not true for general $\pi$. We avoid some complications for the generality in our study for branching laws.)


\begin{theorem} \label{thm fully faith product}
Let $d,m$ be positive integers, and let $\rho \in \mathrm{Irr}^{u,c}(G_k)$. Let 
\[ \mathcal C = \mathcal C_{u_{\rho}(m,d)} . \]
Then the functor $\times_{u_{\rho}(m,d), \mathcal C}$ is fully-faithful.
\end{theorem}

\begin{proof}
It suffices to check the conditions in Lemma \ref{lem essentially injective} in Appendix A. It follows from definition that $\mathrm{Alg}_{\mathcal C}(G_k)$ is Serre.  Condition (1) is automatic. Condition (2) follows from Theorem \ref{thm preserve extension}. Conditions (3) and (4) follow from Lemma \ref{lem speh irred} \cite{LM16}.
\end{proof}

As mentioned in introduction, a key input for the above result is the irreducibility of parabolic induction due to Lapid-M\'inguz \cite{LM16}. It is possible to modify the proof of Theorem \ref{thm preserve extension} to give another proof of Theorem \ref{thm fully faith product} without deducing from the length $2$ case while the length $2$ case is simpler.

Let $p=n_{\rho}md$. For $\pi \in \mathrm{Alg}_{\mathcal C}(G_{n+p})$, define $R_{u_{\rho}(m,d)}(\pi)=\mathrm{Hom}_{G_p}(u_{\rho}(m,d), \pi_{N_n^-})$, which is regarded as a $G_n$-representation by $(g.f)(u)=\mathrm{diag}(1,g).(f(u))$, and is an object in $\mathrm{Alg}_{\mathcal C}(G_n)$. Using adjointness, one checks that $\times_{u_{\rho}(m,d)}$ is left adjoint to $R_{u_{\rho}(m,d)}$.

\begin{corollary} \label{cor adjoint}
Let $u=u_{\rho}(m,d)$. Let $\pi$ be in $\mathrm{Alg}_{\mathcal C}(G_n)$. Then 
\[\pi  \cong R_{u_{\rho}(m,d)}(u_{\rho}(m,d)\times \pi).\]
\end{corollary}

\begin{proof}
Since $R_{u_{\rho}(m,d)}$ is right adjoint to $\times_{u_{\rho}(m,d)}$, Theorem \ref{thm fully faith product} implies that $R_{u_{\rho}(m,d)} \circ \times_{u_{\rho}(m,d)}$ is isomorphic to the identity functor (see e.g. \cite[Lemma 4.24.3]{St}). 
\end{proof}


Corollary \ref{cor adjoint} also gives the following:

\begin{corollary} \label{cor preserving quo sp}
Let $\pi'$ be in $\mathrm{Alg}_{\mathcal C}(G_n)$. Suppose $\pi$ is an irreducible quotient of $u_{\rho}(m,d)\times \pi'$. Then $\pi \cong u_{\rho}(m,d) \times \omega$ for an irreducible quotient $\omega$ of $\pi'$.
\end{corollary}

\begin{proof}
By Frobenius reciprocity, $ \omega' \hookrightarrow R_{u_{\rho}(m,d)}(\pi)$ for some irreducible composition factor $\omega'$ of $\pi'$. Since $\omega'$ is also in $\mathrm{Alg}_{\mathcal C}(G_n)$, we have $\pi \cong u_{\rho}(m,d) \times \omega'$ (Lemma \ref{lem speh irred}). Now, applying the Frobenius reciprocity (or adjointness) on the quotient map from $u_{\rho}(m,d)\times \pi'$ to $\pi\cong u_{\rho}(m,d)\times \omega'$ and Corollary \ref{cor adjoint}, we have a non-zero map from $\pi'$ to $\omega'$, as desired.
\end{proof}

We need a stronger variation for Corollary \ref{cor preserving quo sp}:

\begin{corollary} \label{cor stronger}
Let $\mathcal C$ be as in Theorem \ref{thm fully faith product}. Let $\pi_1$ be a (not necessarily admissible) representation of $G_{n}$. Let $\pi_2$ be in $\mathrm{Alg}_{\mathcal C}(G_{n+p})$, where $p=n_{\rho}md$. Then if $\pi_2$ is a quotient of $u_{\rho}(m,d)\times \pi_1$, 
then there exists a non-zero quotient $\omega$ of $\pi_1$ such that
\[ \pi_2 \cong  u_{\rho}(m,d) \times \omega . \]
In particular, if $\pi_2$ is irreducible, then $ \pi_2 \cong u_{\rho}(m,d) \times \omega $
for an irreducible quotient $\omega$ of $\pi_1$.  If $\pi_2$ is an irreducible Arthur type (resp. unitarizable) representation, then $\pi_2 \cong u_{\rho}(m,d)\times \omega$ for some irreducible Arthur type (resp. unitarizable) representation $\omega$.
\end{corollary}

\begin{proof}
Let $u=u_{\rho}(m,d)$. By adjointness, we have
\begin{align*}
0 \neq     &  \mathrm{Hom}_{G_{n+p}}(u \times \pi_1, \pi_2 ) \cong \mathrm{Hom}_{ G_n}( \pi_1, R_u(\pi_2))  ,
\end{align*}
and let $f$ be the map in $\mathrm{Hom}_{G_n}(\pi_1, R_u(\pi_2))$ corresponding to the surjection from $u_{\rho}(m,d)\times \pi_1$ to $\pi_2$. 

 Now using adjointness, we have the following commutative diagram:
\[ \xymatrix{ \mathrm{Hom}_{G_{n+p}}(u\times \omega, \pi_2)  \ar[d]^{\cong}  & \mathrm{Hom}_{G_{n+p}}(u \times \pi_1, \pi_2) \ar[l] \ar[d]^{\cong} & \mathrm{Hom}_{G_{n+p}}(u\times \tau, \pi_2) \ar[d]^{\cong} \ar[l] &\ar[l] 0 \\
 \mathrm{Hom}_{G_n}(\omega, R_u(\pi_2))  & \mathrm{Hom}_{G_n}(\pi_1, R_u(\pi_2)) \ar[l] & \mathrm{Hom}_{G_n}(\tau, R_u(\pi_2)) \ar[l] & \ar[l] 0}  ,        
\]
where the two horizontal rows are exact from the short exact sequence 
\[0 \rightarrow \omega=\mathrm{ker}~f \rightarrow \pi_1 \rightarrow \tau=\mathrm{im}~f \rightarrow 0 .\]
In particular, we have $\mathrm{im}~f$ embeds to $R_u(\pi_2)$.

The image of the embedding under the leftmost bottom horizontal map is zero by definition and by the commutative diagram, it comes from an element in $\mathrm{Hom}_{G_{n+p}}(u \times \pi_1, \pi_2)$ with zero image by the leftmost top horizontal map. Thus when adjointness back, we get back the surjective map
\[  u \times \tau \rightarrow \pi_2 ,
\]
and the injection
\[  \mathrm{im}~f\cong \tau \hookrightarrow R_u(\pi_2) .
\]
Since $\pi_2$ is in $\mathrm{Alg}_{\mathcal C}(G_{n+p})$, $\tau$ is also in $\mathrm{Alg}_{\mathcal C}(G_n)$. Thus the first surjection implies that the number of composition factors in $\pi_2$ is at most that of $\tau$ by Lemma \ref{lem speh irred}. By Corollary \ref{cor adjoint}, for each irreducible $\pi' \in \mathrm{Alg}_{\mathcal C}(G_{n+p})$, $R_u(\pi')$ is either irreducible or zero. Thus with the fact that $R_u$ is a left exact functor, the number of composition factors of $\pi_2$ is at least that of composition factors of $R_u(\pi_2)$. Hence, the second injection implies that the number of composition factors in $\pi_2$ is at least that of $\tau$. This implies the coincidence on the numbers and so the surjection must be an isomorphism i.e. $ u \times \tau \cong \pi_2$.

It remains to prove the last statement. Suppose $\pi_2 \cong u_{\rho}(m,d) \times \omega$ is an Arthur type representation. Then $\pi_2$ and $u_{\rho}(m,d)$ being Hermitian self-dual implies that 
\[ \bar{\omega}^{\vee} \times u_{\rho}(m,d) \cong \bar{\pi}_2^{\vee} \cong \pi_2 \cong u_{\rho}(m,d) \times \omega \cong \omega \times u_{\rho}(m,d) .
\]
This implies that $\bar{\omega}^{\vee} \cong \omega$ by Lemma \ref{lem speh irred} and so it is Hermitian self-dual. Thus $\omega$ is unitarizable by a result of Bernstein \cite[Corollary 8.2]{Be84}. Now the classification \cite{Ta86} of unitarizable representations and unique factorization give that $\omega$ is an Arthur type representation. The proof for the assertion for unitarizable representation is similar.

\end{proof}

\subsection{Generalizations} \label{ss generalize}

While our result of Theorem \ref{thm fully faith product} is for a special class of representations, one can generalize to a larger class of examples as long as an analogue of Theorem \ref{thm preserve extension} is established. 

In \cite[Proposition 5.1]{LM16}, it describes a criteria which $\pi \times \langle \Delta \rangle$ is irreducible for an irreducible representation $\pi$ and a segment $\Delta$. For a fixed irreducible representation $\pi$, let $\mathrm{Alg}_{\pi}(G_n)$ be the full subcategory of $\mathrm{Alg}(G_n)$ which contains objects of finite length with simple composition factors $\tau$ satisfying the property that $\pi \times \langle \Delta \rangle$ for any segment $\Delta$ in the associated multisegment $\mathfrak m$ of $\pi$. We expect that if $\tau$ is $\mathrm{Alg}_{\pi}(G_n)$ is an indecomposable representation of length $2$, then $\pi \times \tau$ is also indecomposable of length $2$.

\section{Appendix A: Some homological algebra} \label{sect homo alg}


Let $\mathcal A=\mathrm{Alg}(G_l)$. Let $\mathcal B=\mathrm{Alg}(G_n)$. Via Yoneda extension,  any element in $\mathrm{Ext}^1_{\mathcal A}(X,Y)$ corresponds to a short exact sequence in $\mathcal A$, and zero element corresponds to the split sequence. Then, for an additive exact functor $\mathcal F$, $\mathcal F$ sends a short exact sequence to a short exact sequence, and this defines a map from $\mathrm{Ext}^1_{\mathcal A}(X,Y)$ to $\mathrm{Ext}^1_{\mathcal B}(\mathcal F(X),\mathcal F(Y))$. 




\begin{lemma} \label{lem essentially injective}
Let $\mathcal C$ be a full Serre subcategory of $\mathcal A=\mathrm{Alg}(G_l)$. Let $\mathcal B=\mathrm{Alg}(G_n)$ and let $\mathcal D$ be a Serre full subcategory of $\mathcal B$. Let $\mathcal F: \mathcal C \rightarrow \mathcal D$ be an exact additive functor. We also regard objects in $\mathcal C$ as objects in $\mathcal A$ via the inclusion. Assume that
\begin{enumerate}
\item any object in $\mathcal C$ is of finite length;
\item for any simple objects $X, Y$ in $\mathcal C$, the induced map of $\mathcal F$, from $\mathrm{Ext}^1_{\mathcal A}(X, Y)$ to $\mathrm{Ext}^1_{\mathcal B}(\mathcal F(X), \mathcal F(Y))$ is an injection and,
\item $\mathcal F(X)$ is a simple object in $\mathcal D$ if $X$is simple in $\mathcal C$; and
\item  for any simple objects $X$ and $Y$ in $\mathcal C$, $\mathcal F(X) \cong \mathcal F(Y)$ if and only if $X \cong Y$.
\end{enumerate}
 Then for any objects $X, Y$ in $\mathcal C$, the induced map from $\mathrm{Ext}^1_{\mathcal A}(X,Y)$ to $\mathrm{Ext}^1_{\mathcal B}(\mathcal F(X), \mathcal F(Y))$ is also injective, and $\mathcal F: \mathcal C \rightarrow \mathcal D$ is fully-faithful i.e. 
\[\mathrm{Hom}_{\mathcal B}(\mathcal F(X), \mathcal F(Y))\cong \mathrm{Hom}_{\mathcal D}(\mathcal F(X), \mathcal F(Y)) \cong \mathrm{Hom}_{\mathcal C}(X,Y) \cong \mathrm{Hom}_{\mathcal A}(X,Y) \]
 for any objects $X,Y$ in $\mathcal C$. 

\end{lemma}

\begin{proof}
Let $X$ and $Y$ be objects in $\mathcal C$. When both lengths of $X$ and $Y$ are $1$ in $\mathcal C$, 
\[  \mathrm{Hom}_{\mathcal D}(\mathcal F(X), \mathcal F(Y)) \cong \mathrm{Hom}_{\mathcal C}(X,Y), \quad \mathrm{Ext}^1_{\mathcal A}(X,Y) \hookrightarrow \mathrm{Ext}^1_{\mathcal B}(\mathcal F(X), \mathcal F(Y))
\]
are guaranteed by (2), (3) and (4). We first fix the length of $X$ to be at most some $n$. We shall prove the statement for arbitrary $Y$ by induction on the length of $Y$.

For an object $Y$ in $\mathcal C$, let $Y_1$ be an irreducible quotient of $Y$. Then we have a short exact sequence:
\[  0 \rightarrow Y_2 \rightarrow Y \rightarrow Y_1 \rightarrow 0 .
\]
Since $\mathcal C$ is Serre, $Y_1$ and $Y_2$ are in $\mathcal C$.

Note that we have the following commutative diagram:
\[ \xymatrix{ \mathrm{Hom}_{\mathcal A}(X, Y_1) \ar[r] \ar[d] &  \mathrm{Ext}^1_{\mathcal A}(X, Y_2) \ar[r] \ar[d] & \mathrm{Ext}^1_{\mathcal A}(X, Y) \ar[r] \ar[d] & \mathrm{Ext}^1_{\mathcal A}(X, Y_1) \ar[d] \\
 \mathrm{Hom}_{\mathcal B}(\mathcal F(X), \mathcal F(Y_1)) \ar[r] & \mathrm{Ext}^1_{\mathcal B}(\mathcal F(X), \mathcal F(Y_2)) \ar[r] & \mathrm{Ext}^1_{\mathcal B}(\mathcal F(X), \mathcal F(Y)) \ar[r] & \mathrm{Ext}^1_{\mathcal B}(\mathcal F(X), \mathcal F(Y_1)) 
},
\]
where the horizontal maps come from long exact sequences, in which the connecting homomorphism is the Yoneda product, and vertical maps for $\mathrm{Ext}^1$ are described in the beginning of this section, and the vertical map for $\mathrm{Hom}$ is the map induced from the functor.

We have the first vertical arrow is isomorphism and the second and forth vertical arrows are injections by induction hypothesis. Then it is direct to check that the third vertical arrow is also an injection. 

Now we consider another commutative diagram:
\[ \xymatrix{ 0 \ar[r] & \mathrm{Hom}_{\mathcal A}(X,Y_1) \ar[r] \ar[d] & \mathrm{Hom}_{\mathcal A}(X, Y) \ar[r] \ar[d] & \mathrm{Hom}_{\mathcal A}(X,Y_2) \ar[r] \ar[d] & \mathrm{Ext}^1_{\mathcal A}(X,Y_1) \ar[d] \\
0 \ar[r] & \mathrm{Hom}_{\mathcal B}(\mathcal F(X),\mathcal F(Y_1)) \ar[r] & \mathrm{Hom}_{\mathcal B}(\mathcal F(X), \mathcal F(Y)) \ar[r] & \mathrm{Hom}_{\mathcal B}(\mathcal F(X), \mathcal F(Y_2)) \ar[r] & \mathrm{Ext}^1_{\mathcal B}(\mathcal F(X),\mathcal F(Y_1))  
}
\]
The first and third vertical arrows are isomorphisms by induction and the last vertical arrow is an injection by induction again. Thus we have that the second vertical arrow is an isomorphism. 

Now we switch the role of $X$ and $Y$, and use similar argument to prove that the assertion is true for $X$ and $Y$ of arbitrary finite length.

\end{proof}

\begin{remark}
The above lemma is also valid for arbitrary abelian categories $\mathcal A$ and $\mathcal B$ which are Schurian $k$-categories, where $k$ is a field i.e. 
\[  \mathrm{Hom}_{\mathcal A}(X,X) \cong k, \quad \mbox{ and } \mathrm{Hom}_{\mathcal B}(Y,Y) \cong k 
\] 
for any simple objects $X$ and $Y$ in $\mathcal A$ and $\mathcal B$ respectively.
\end{remark}

 \end{document}